\documentclass[psamsfonts,11pt]{amsart}

\usepackage[left=1in,right=1in,top=1.00in,bottom=1.00in]{geometry}

\usepackage{amssymb,mathrsfs}
\newtheorem{theorem}{Theorem}[section]
\newtheorem{lemma}[theorem]{Lemma}
\newtheorem{corollary}[theorem]{Corollary}
\newtheorem{proposition}[theorem]{Proposition}

\theoremstyle{definition}
\newtheorem{definition}[theorem]{Definition}
\newtheorem{example}[theorem]{Example}

\newtheorem{remark}[theorem]{Remark}

\theoremstyle{remark}
\newtheorem{claim}{Claim}

\numberwithin{equation}{section}

\newcommand\N{\mathbb{N}}
\newcommand\Z{\mathbb{Z}}
\newcommand\T{\mathbb{T}}
\newcommand\cont{\mathfrak{c}}

\newcommand\du[1]{\widehat{#1}}
\newcommand\Prm{\mathbb{P}}

\title[]{Subgroups of direct products closely approximated by direct sums}

\author[]{M. Ferrer}
\address[M. Ferrer]{Universitat Jaume I, Instituto de Matem\'aticas de Castell\'on,
Campus de Riu Sec, 12071 Castell\'{o}n, Spain.}
\email{mferrer@mat.uji.es}

\author[]{S. Hern\'andez}
\address[S. Hern\'andez]{Universitat Jaume I, Departamento de Matem\'{a}ticas,
Campus de Riu Sec, 12071 Castell\'{o}n, Spain.}
\email{hernande@mat.uji.es}

\author[]{D. Shakhmatov}
\address[D. Shakhmatov]{Division of Mathematics, Physics and Earth Sciences\\
Graduate School of Science and Engineering\\
Ehime University\\
Matsuyama 790-8577\\
Japan}
\email{dmitri.shakhmatov@ehime-u.ac.jp}
\thanks{The first and second named authors were partially support by
Universitat Jaume I, grant P1·1B2012-05. The third named author was partially 
supported by the Grant-in-Aid for Scientific Research~(C) No.~22540089 by the 
Japan Society for the Promotion of Science (JSPS)}

\begin{document}
\begin{abstract}
Let $I$ be an infinite set, $\{G_i:i\in I\}$ be a family of (topological) groups and
$G=\prod_{i\in I} G_i$ be its direct product.
For $J\subseteq I$, $p_{J}: G\to \prod_{j\in J} G_j$ denotes the projection.
We say that a subgroup $H$ of
$G$
is: (i)~\emph{uniformly controllable\/}
in $G$
provided that
for every finite set $J\subseteq I$
there exists a finite set
$K\subseteq I$ such that
$p_{J}(H)=p_{J}(H\cap\bigoplus_{i\in K} G_i)$;
(ii)~\emph{controllable\/}
in $G$ provided that
$p_{J}(H)=p_{J}(H\cap\bigoplus_{i\in I} G_i)$
for every finite set $J\subseteq I$;
(iii)~\emph{weakly controllable\/} in $G$
if $H\cap \bigoplus_{i\in I} G_i$ is dense in $H$, when $G$ is equipped with the Tychonoff product topology.
One easily proves that (i)$\to$(ii)$\to$(iii).
We thoroughly investigate the question as to when these two arrows can be
reversed. We prove that the first arrow can be reversed when $H$ is compact,
but the second arrow cannot be reversed even when $H$ is compact.
Both arrows can be reversed if all groups $G_i$ are finite.
When $G_i=A$ for all $i\in I$, where $A$ is an abelian group, we show that the first arrow can be reversed for {\em all\/} subgroups $H$ of $G$ if and only if
$A$ is finitely generated.
Connections with coding theory are highlighted.
\end{abstract}
\date{May 17, 2013}



\maketitle

\section{Three ways to embed a group into a direct product}

Let $\{G_i:i\in I\}$ be a family of groups. As usual, its {\em direct product\/}
$G=\prod_{i\in I} G_i$ is the set of all functions
$g: I\to \bigcup\{G_i:i\in I\}$ such that $g(i)\in G_i$ for every $i\in I$.
The group operation on $G$ is defined coordinate-wise: the product
$gh\in G$ of $g\in G$ and $h\in G$ is the function defined by
$gh(i)=g(i)h(i)$ for each $i\in I$.
Clearly,
the identity element $1$ of $G$ is the function that assigns the identity element of $G_i$ to every $i\in I$.
The subgroup
$$
\bigoplus_{i\in I} G_i=\{g\in G: g(i)=1
\mbox{ for all but finitely many }
i\in I\}
$$
of $G$
is called the {\em direct sum\/} of the family $\{G_i:i\in I\}$.
For $J\subseteq I$, the projection
$p_{IJ}: \prod_{i\in I} G_i\to \prod_{j\in J} G_j$
defined by $p_{IJ}(g)=g\restriction_J$ for $g\in G$, is the group
homomorphism.
When $G_i=M$ for every $i\in I$, then we write $M^I$ instead of $\prod_{i\in I} G_i$.

Our first definition introduces two group-theoretic notions that
characterize the way a group is embedded into a direct product of
groups.

\begin{definition}
\label{def:controllability}
Let $I$ be a set, $\{G_i:i\in I\}$ be a family of groups and
$G=\prod_{i\in I} G_i$ be its direct product.
We say that a subgroup $H$ of
$G$
is:
\begin{itemize}
\item[(i)] \emph{controllable\/}
in $G$ provided that
$p_{IJ}(H)=p_{IJ}(H\cap\bigoplus_{i\in I} G_i)$
for every finite set $J\subseteq I$;

\item[(ii)]
\emph{uniformly controllable\/}
in $G$
provided that
for every finite set $J\subseteq I$
there exists a finite set
$K\subseteq I$ such that
$p_{IJ}(H)=p_{IJ}(H\cap\bigoplus_{i\in K} G_i)$.
\end{itemize}
\end{definition}

Clearly,
a uniformly controllable subgroup is controllable.
The following proposition gives two important instances
of (uniformly) controllable subgroups of direct products.

\begin{proposition}
\label{examples:of:controllability}
Let $\{G_i:i\in I\}$ be a family of
groups and $G=\prod_{i\in I} G_i$.
\begin{itemize}
\item[(i)]
Every subgroup of $G$ containing $\bigoplus_{i\in I} G_i$ is
uniformly
controllable in $G$.
\item[(ii)]
Every subgroup of $\bigoplus_{i\in I} G_i$ is controllable in $G$.
\end{itemize}
\end{proposition}
\begin{proof}
(i)
Let
$H$ be a subgroup of $G$ such that $\bigoplus_{i\in I} G_i\subseteq H$.
Let $J\subseteq I$ be a finite set.
Then
$\bigoplus_{i\in J} G_i\subseteq \bigoplus_{i\in I} G_i\subseteq H$,
and so
$$
\bigoplus_{i\in J} G_i=p_{IJ}\left(\bigoplus_{i\in I} G_i\right)\subseteq p_{IJ}(H)\subseteq p_{IJ}\left(\prod_{i\in I} G_i\right)=\bigoplus_{i\in J} G_i,
$$
which yields
$p_{IJ}(H)=\bigoplus_{i\in J} G_i=p_{IJ}(\bigoplus_{i\in J} G_i)=
p_{IJ}(H\cap\bigoplus_{i\in J} G_i)$.
Therefore, $K=J$ satisfies item (ii) of Definition \ref{def:controllability}.

(ii)
If $H$ is a subgroup of $\bigoplus_{i\in I} G_i$, then
$H=H\cap \bigoplus_{i\in I} G_i$, and so item (i) of Definition \ref{def:controllability} trivially holds.
\end{proof}

For future proofs, it would be helpful to restate Definition \ref{def:controllability}
without
using the language of projections:
\begin{proposition}
\label{reformulation:of:definition}
Let $\{G_i:i\in I\}$ be a family of groups and
 $H$ be a subgroup of its direct product $G=\prod_{i\in I} G_i$.
Then:
\begin{itemize}
\item[(i)]
$H$ is controllable in $G$ if and only if
for every
$h\in H$
and each finite set $J\subseteq I$ there exists
$g\in H\cap\bigoplus_{i\in I} G_i$
such that
$g\restriction_J=h\restriction_J$;
\item[(ii)]
$H$ is uniformly controllable in $G$ if and only if
for every finite set $J\subseteq I$ there exists
a finite set $K\subseteq I$
such that
for
every
$h\in H$
one can find
$g\in H\cap\bigoplus_{i\in K} G_i$
with
$g\restriction_J=h\restriction_J$.
\end{itemize}
\end{proposition}

When all groups $G_i$ in question have some topology, we
always
equip their direct product
$\prod_{i\in I} G_i$ with
the Tychonoff product topology,
and
we also consider the following topological property.

\begin{definition}
\label{def:weak:controllability}
Let $I$ be a set, $\{G_i:i\in I\}$ be a family of topological groups
and
$G=\prod_{i\in I} G_i$ be its
direct product.
We say that a subgroup $H$ of
$G$
is
\emph{weakly controllable\/} in $G$ if $H\cap \bigoplus_{i\in I} G_i$ is dense in $H$.
\end{definition}

The relevance of the topological notion from Definition \ref{def:weak:controllability} to the group-theoretic notions
from Definition \ref{def:controllability} can be seen from the following
proposition which justifies the use of the word ``weakly'' in
Definition \ref{def:weak:controllability}.

\begin{proposition}
\label{controllable:is:weakly:controllable}
A controllable subgroup
of an arbitrary direct product
$G=\prod_{i\in I} G_i$ of topological groups $G_i$ is weakly controllable.
\end{proposition}
\begin{proof}
Let $H$ be a controllable subgroup of $G$.
We
need
to show that $H\cap \bigoplus_{i\in I} G_i$ is dense in $H$.
Let $O$ be an open subset of $G$ with $O\cap H\not=\emptyset$.
It suffices to check that $O\cap H\cap \bigoplus_{i\in I} G_i\not=\emptyset$.
Fix
$h\in O\cap H$.
By
the definition of the product topology,
we can find a finite set $J\subseteq I$
and an open neighbourhood $U_i$ of
$h(i)$
in $G_i$ for every $i\in J$
such that
\begin{equation}
\label{eq:U:O}
h\in W=\{g\in G: g(i)\in U_i
\mbox{ for all }
i\in J\}\subseteq O.
\end{equation}
Since $H$ is controllable in $G$,
Proposition \ref{reformulation:of:definition}(i) allows us to find
$g\in H\cap\bigoplus_{i\in I} G_i$
with
$g\restriction_J=h\restriction_J$.
From
this
and \eqref{eq:U:O}
it follows that
$g\in W\subseteq O$.
Therefore,
$g\in O\cap H\cap \bigoplus_{i\in I} G_i\not=\emptyset$.
\end{proof}

It is clear from Definition \ref{def:controllability}
that a uniformly controllable subgroup is controllable. Combining this with
the last proposition, we obtain the following chain of implications:
\begin{equation}
\label{three:notions}
\mbox{uniformly controllable}
\to
\mbox{controllable}
\to
\mbox{weakly controllable}.
\end{equation}

These implications show that the three notions introduced above express
a degree of how closely a subgroup of a direct product is approximated by its
direct sum.

In this paper we study the question whether
the
two
implications
above
can be reversed
for various classes of groups.

The principal result in Section \ref{Positive:results:section} is
Theorem \ref{controllable:and:uniformly:controllable:coincide:for:compact:subgroups} which asserts that the two notions from Definition \ref{def:controllability}
coincide for compact subgroups of arbitrary direct products $\prod_{i\in I} G_i$
of topological groups $G_i$.
Since both notions from Definition \ref{def:controllability} are purely algebraic
in the sense that they do not depend on topologies of groups $G_i$,
it is somewhat surprising that a topological property such as compactness
imposed on a subgroup of $\prod_{i\in I} G_i$
has an influence on
the invertibility of the first arrow in \eqref{three:notions}.
In particular, this arrow is reversible for closed subgroups of arbitrary
products $\prod_{i\in I} G_i$ of compact groups $G_i$; see Corollary
\ref{controllable:and:uniformly:controllable:coincide:in:closed:subgroups:of:compact:products}.
When all groups $G_i$ are finite, then both arrows in \eqref{three:notions}
can be reversed for arbitrary (not necessarily closed) subgroups of the product $\prod_{i\in I} G_i$; see Corollary \ref{uniform:controllability}.
This result has profound applications in coding theory.

The invertibility of the first arrow in \eqref{three:notions}
is thouroughly investigated in Section
\ref{Controllable:non-uniformly:controllable}.
In Theorem
\ref{characterization:of:powers:in:which:all:controllable:subgroups:are:uniformaly:controllable}
we characterize abelian groups $M$ such that
every controllable subgroup of an infinite power $M^I$ is uniformly controllable;
this property holds if and only if $M$ is finitely generated.
It follows that, for a non-finitely generated abelian group $M$ and an infinite set $I$, the product $M^I$ always contains some subgroup that is controllable but not uniformly controllable. As an application, it follows that the first arrow
in \eqref{three:notions} is not reversible for arbitrary subgroups of
the countable power $M^\N$ of the compact metric group $M=\Z(2)^\N$ of order $2$ (Example \ref{power:of:Z(2)}),
thereby demonstrating that
compactness of the subgroup is essential in Theorem
\ref{controllable:and:uniformly:controllable:coincide:for:compact:subgroups}
and closedness of the subgroup is essential in Corollary
\ref{controllable:and:uniformly:controllable:coincide:in:closed:subgroups:of:compact:products}. We push this even further in
Corollary \ref{lots:of:controllable:not:uniformly:controllable:subgroups}
by constructing a large family (having the maximal size $2^\cont$) of subgroups of the countable
power $\T^\N$ of the circle group $\T$ each of which is controllable but not uniformly controllable. (Here $\cont$ denotes the cardinality of the continuum.)
Both Example \ref{power:of:Z(2)} and Corollary \ref{lots:of:controllable:not:uniformly:controllable:subgroups} demonstrate that finiteness of groups $G_i$ in
Corollary \ref{uniform:controllability} cannot be replaced by their compactness.

In Section \ref{weakly:controllable:non:controllable}
we thouroughly investigate
the invertibility of the second arrow in \eqref{three:notions}.
In Theorem \ref{many:weakly:controllable:not:controllable;subgroups}
we prove that the family $\mathscr{H}$ of all weakly controllable
non-controllable subgroups of $\T^\N$ has cardinality $2^\cont$, which is the maximal
size possible.
Furthermore, we exhibit a compact member of $\mathscr{H}$ in
Example \ref{compact:subgroup}
and
a countable torsion member of $\mathscr{H}$ in Example \ref{torsion:subgroup}.

In Section \ref{sec:4} we study the behaviour of the three notions under taking closures and dense subgroups.
The three notions of controllability introduced in this section are closely
related to coding theory,
and this connection is explained in detail in Section
\ref{coding:section}. Among other things, this close connection justifies our choice of terminology in Definitions \ref{def:controllability}(i) and \ref{def:weak:controllability}. To the best of our knowledge,
the notion of uniform controllability introduced in Definition \ref{def:controllability}(ii) in new and has no analogue in coding theory, although it is weaker then the classical notion of strong controllability; see  Definition \ref{classical:definition}(iv) and implications in \eqref{strongly:controllable:implications}.
Theorem \ref{unif:but:not:strongly} shows that, for every prime number $p$,
the family $\mathscr{H}_p$ of subgroups of the product $\mathbb{Z}(p)^\N$ which are uniformly controllable but not strongly controllable
has cardinality $2^\cont$, the maximal size possible.

In
Section \ref{profinite:section} we study
the structure of profinite metric abelian groups.
We prove that
a compact metric profinite abelian group $G$ whose {\em torsion part\/}
$$
t(G)=\{x\in G: nx=0
\ \mbox{ for some }
\
n\in\N\setminus\{0\}\}
$$
is dense in $G$ is
topologically isomorphic to a product of countably many finite cyclic groups;
see Theorem \ref{lemma:2}.
In particular, a topological abelian group $G$ is topologically isomorphic to a direct product of countably many finite cyclic groups if and only if
 $G$ is zero-dimensional, compact metric and $t(G)$ is dense in $G$;
see Corollary \ref{cor:8.3}.
As a corollary, we show that
a closed weakly controllable subgroup of a countable direct product of finite
abelian groups is itself topologically isomorphic to a product of finite (cyclic) abelian groups; see Corollary \ref{product:corollary}.
These results should be compared with similar results for the class of Valdivia
compact groups
obtained recently in \cite{Kubis:2008,
Chigogidze:2008,KubisUspenskij:2005}.

Experts in coding theory may want to start with Sections
\ref{coding:section} and \ref{conclusion:section} first, and then
proceed with the rest of the paper.

\section{Cases when  various forms of controllability coincide}

\label{Positive:results:section}

In this section we investigate special cases when arrows in \eqref{three:notions} can be reversed.

Recall that a group $G$ satisfies the {\em ascending chain condition\/}
provided that every ascending chain
$G_0\subseteq G_1\subseteq \dots\subseteq G_n\subseteq G_{n+1}\subseteq\dots$ of subgroups of $G$ stabilizes; that is, there exists $k\in\N$ such that $G_k=G_m$ for all $m\in\N$ with $m\ge k$.

The
first arrow in \eqref{three:notions}
can be reversed when all groups $G_i$ satisfy the  ascending chain condition.
\begin{proposition}
\label{all:groups:satisfy:the:ascending:chain:condition}
Let $I$ be a set and let $\{G_i:i\in I\}$ be a family of
groups such that all $G_i$ satisfy the ascending chain condition.
Then
a subgroup of the direct product
$G=\prod_{i\in I} G_i$ is controllable in $G$ if and only if it is uniformly controllable in $G$.
\end{proposition}
\begin{proof}
Clearly, every uniformly controllable subgroup of $G$ is controllable in $G$.
To prove the reverse implication, assume that $H$ is a controllable subgroup of $G$.
Fix a finite subset $J$ of $I$.
Since each $G_i$ for $i\in J$ satisfies the ascending chain condition, so does
the finite product $\prod_{i\in J} G_i$.
Therefore,
the subgroup $p_{IJ}(H)$ of $\prod_{i\in J} G_i$ is finitely generated.
Let $X$ be a finite set of generators for $p_{IJ}(H)$.
Since $H$ is controllable in $G$,
we have
$p_{IJ}(H)=p_{IJ}(H\cap\bigoplus_{i\in I} G_i)$,
so
we can fix
a finite set $Y\subseteq H\cap\bigoplus_{i\in I} G_i$ with
$X=p_{IJ}(Y)$.
Since $Y\subseteq \bigoplus_{i\in I} G_i$, there exists a finite set
$K\subseteq I$ such that $Y\subseteq \bigoplus_{i\in K} G_i$.
Now
$$
p_{IJ}(H)=\langle X \rangle =
\langle p_{IJ}(Y)\rangle \subseteq p_{IJ}\left(H\cap \bigoplus_{i\in K} G_i\right)
\subseteq
p_{IJ}(H),
$$
which yields
$p_{IJ}(H)=p_{IJ}(H\cap \bigoplus_{i\in K} G_i)$.
This proves that $H$ is uniformly controllable in $G$.
\end{proof}

The
second arrow in \eqref{three:notions}
can be reversed when all groups $G_i$ are discrete.
\begin{proposition}
\label{controllability}
Let $I$ be a set and let $\{G_i:i\in I\}$ be a family of
discrete groups. Then
a subgroup $H$ of the direct product
$G=\prod_{i\in I} G_i$ is weakly controllable in $G$
if and only if it is controllable in $G$.
\end{proposition}
\begin{proof}
The ``if'' part is proved in Proposition \ref{controllable:is:weakly:controllable}. To prove the ``only if'' part,
suppose that $H$ is weakly controllable in $G$.
To show that $H$ is controllable in $G$,
fix
$h\in H$
and a finite set $J\subseteq I$.
Since all $G_i$ are discrete, it follows from the definition of the product topology that
$U=\{g\in G: g(i)=h(i)$ for all $i\in J\}$
is an open subset of $G$.
Note that $h\in U\cap H$,
so $U\cap H$ is a non-empty open subset of $H$.
By the weak controllability of $H$, we can find
$g\in U\cap H\cap \bigoplus_{i\in I} G_i$.
Clearly, $g\restriction_J=h\restriction_J$. Therefore,
$H$ is controllable in $G$
by Proposition
\ref{reformulation:of:definition}(i).
\end{proof}

When the groups in question are finite, Proposition
 \ref{controllability}
can be strengthened a bit further, allowing the reversal of both arrows in
\eqref{three:notions}.

\begin{corollary}
\label{uniform:controllability}
Let $I$ be a set and let $\{G_i:i\in I\}$ be a family of
finite (discrete) groups. Then
for every subgroup $H$ of the direct product
$G=\prod_{i\in I} G_i$ the following conditions are equivalent:
\begin{itemize}
\item[(i)] $H$ is weakly controllable;
\item [(ii)] $H$ is controllable;
\item [(iii)] $H$ is uniformly controllable.
\end{itemize}
\end{corollary}
\begin{proof}
The equivalence (i)$\leftrightarrow$(ii) is proved in Proposition \ref{controllability},
and
the equivalence (ii)$\to$(iii)
is proved in Proposition
\ref{all:groups:satisfy:the:ascending:chain:condition}.
\end{proof}

The rest of this section is devoted to showing that the first arrow in
\eqref{three:notions} can be reversed for
compact subgroups of direct products; see
Theorem
\ref{controllable:and:uniformly:controllable:coincide:for:compact:subgroups}.
For this end, we shall need a general proposition which has its own interest.

\begin{proposition}
\label{preservation:by:projections}
Let $\{G_i:i\in I\}$ be a family of (topological) groups and let $H$ be a subgroup
of its direct product $G=\prod_{i\in I} G_i$.
Let $S$ be a subset of $I$ and
$G_S=\prod_{i\in S} G_i=p_{IS} (G)$.
\begin{itemize}
\item[(i)]
If $H$ is (uniformly) controllable in $G$, then $p_{IS}(H)$ is
(uniformly) controllable in $G_S$.
\item[(ii)]
If $H$ is weakly controllable in $G$, then $p_{IS}(H)$ is
weakly controllable in $G_S$.
\end{itemize}
\end{proposition}
\begin{proof}
(i)
For every $K\subseteq I$, one has
\begin{equation}
\label{eq:star}
p_{IS}\left(H\cap \bigoplus_{i\in K} G_i\right)
\subseteq
p_{IS}(H)\cap \bigoplus_{i\in K\cap S} G_i.
\end{equation}
Let $J\subseteq S$ be a finite set. Since $H$ is (uniformly) controllable,
$p_{IJ}(H)=p_{IJ}(H\cap \bigoplus_{i\in K} G_i)$ holds for some $K\subseteq I$,
 where $K=I$ when $H$ is controllable and $K$ is finite when $H$ is uniformly controllable.
Since $p_{IJ}=p_{SJ}\circ p_{IS}$, from this and
\eqref{eq:star} we get
$$
p_{SJ}(p_{IS}(H))=p_{SJ}\left(p_{IS}\left(H\cap \bigoplus_{i\in K} G_i\right)\right)
\subseteq
p_{SJ}\left(p_{IS}(H)\cap \bigoplus_{i\in K'} G_i\right)
\subseteq
p_{SJ}(p_{IS}(H)),
$$
where $K'=K\cap S$.
This
yields
$p_{SJ}(p_{IS}(H))=p_{SJ}\left(p_{IS}(H)\cap \bigoplus_{i\in K'} G_i\right)$.
When $K$ is finite, $K'$ is also finite, and when $K=I$,
$K'=S$.
This shows that $p_{IS}(H)$ is (uniformly) controllable in $G_S$.

In the proof of (ii) one uses the fact that the projection
$p_{IS}:G\to G_S$ is continuous, and so it preserves density.
\end{proof}

We need to prove the special case of
Theorem
\ref{controllable:and:uniformly:controllable:coincide:for:compact:subgroups}
first.

\begin{lemma}
\label{compact:in:countable:products}
Let $I$ be a countable set and $\{G_i:i\in I\}$ be a family of topological groups.
Then every compact controllable subgroup of $G=\prod_{i\in I} G_i$ is uniformly controllable.
\end{lemma}
\begin{proof}
Let $H$ be a compact controllable subgroup of $G$.
Fix a finite subset $J$ of $I$.
We need to find a finite set $K\subseteq I$
such that
$p_{IJ}(H)=p_{IJ}(H\cap \bigoplus_{i\in K} G_i)$.

For every finite subset $F$ of $I$,
$
H_F=H\cap \bigoplus_{i\in F} G_i
$
is a closed (and thus, compact) subgroup of $H$, so
its continuous homomorphic image
$p_{IJ}(H_F)$ is a compact (and thus, closed) subgroup of
$p_{IJ}(H)$.
The latter group is compact as well, as a continuous image of the compact group $H$.
Note that
$H\cap \bigoplus_{i\in I} G_i= \bigcup\{H_F: F\subseteq I$ is finite$\}$.
Since $H$ is controllable,
\begin{equation}
\label{eq:projection}
p_{IJ}(H)=
p_{IJ}\left(H\cap \bigoplus_{i\in I} G_i\right)=
\bigcup\{p_{IJ}(H_F): F\subseteq I
\mbox{ is finite}
\}.
\end{equation}
Since the collection of finite subsets of the countable set $I$ is countable,
applying \eqref{eq:projection} and the Baire category theorem to $p_{IJ}(H)$,
we can find a finite set $F^*\subseteq I$
such that $N=p_{IJ}(H_{F^*})$ has non-empty interior in $p_{IJ}(H)$.
Since $N$ is a subgroup of $p_{IJ}(H)$, it must be open in
$p_{IJ}(H)$.
Since the latter group is compact, $N$ has finite index in
$p_{IJ}(H)$;
that is, there exists a finite set $Y\subseteq p_{IJ}(H)$ such that
$p_{IJ}(H)=YN$.
Applying \eqref{eq:projection}
we can find a finite set $X\subseteq H\cap \bigoplus_{i\in I} G_i$ with
$Y=p_{IJ}(X)$.
Finally, let $K$ be a finite subset of $I$ such that $F^*\subseteq K$ and
$X\subseteq \bigoplus_{i\in K} G_i$.

Let
$h\in H$ be arbitrary.
Since $p_{IJ}(h)\in p_{IJ}(H)=YN$, there exists $y\in Y$ and $z\in N$ with
$p_{IJ}(h)=yz$.
Pick $x\in X\subseteq H\cap\bigoplus_{i\in K} G_i$ and $h^*\in H_{F^*}\subseteq H\cap\bigoplus _{i\in K} G_i$ such that
$y=p_{IJ}(x)$ and $z=p_{IJ}(h^*)$.
Since $H$ is a subgroup of $G$,
so is $H\cap \bigoplus_{i\in K} G_i$.
Hence,
$h'=xh^*\in H\cap\bigoplus _{i\in K} G_i$.
Finally note that
$p_{IJ}(h')=p_{IJ}(xh^*)=p_{IJ}(x) p_{IJ}(h^*)=yz=p_{IJ}(h)$.
This shows that
$p_{IJ}(H)\subseteq p_{IJ}(H\cap \bigoplus_{i\in K} G_i)$.
The reverse inclusion is obvious.
\end{proof}

\begin{theorem}
\label{controllable:and:uniformly:controllable:coincide:for:compact:subgroups}
Let $\{G_i:i\in I\}$ be a family of topological groups and $G=\prod_{i\in I} G_i$
be its direct product.
For a compact subgroup $H$ of $G$, the following conditions are equivalent:
\begin{itemize}
\item[(i)] $H$ is controllable in $G$;
\item[(ii)] $H$ is uniformly controllable in $G$.
\end{itemize}
\end{theorem}
\begin{proof}
The implication (ii)$\to$(i) is trivial.
To check the implication (i)$\to$(ii), we assume that $H$ is not uniformly controllable in $G$, and we shall prove that $H$ is not controllable in $G$ either.

Since $H$ is not uniformly controllable in $G$, we can fix a finite set $J\subseteq I$ such that
\begin{equation}
\label{eq:not:uniformly}
p_{IJ}(H)\not=p_{IJ}\left(H\cap \bigoplus_{i\in K} G_i\right)
\mbox{ for every finite set }
K\subseteq I.
\end{equation}

For each $F\subseteq I$ define
\begin{equation}
\label{def:of:NF}
N_F=\{h\in H: h(i)=1
\mbox{ for all }
i\in F\}
\end{equation}
and note that $N_F$ is a closed subgroup of $H$.

\begin{claim}
For every finite set $K\subseteq I$ there exists a finite set $F\subseteq I$
disjoint from $K$ such that $p_{IJ}(H)\setminus p_{IJ}(N_F)\not=\emptyset$.
\end{claim}
\begin{proof}
Use \eqref{eq:not:uniformly} to fix $h_K\in H$ such that
$p_{IJ}(h_K)\not\in p_{IJ}\left(H\cap \bigoplus_{i\in K} G_i\right)$.
Then
\begin{equation}
\label{def:Phi:K}
\Phi_K=\{h\in H: p_{IJ}(h)=p_{IJ}(h_K)\}=H\cap p_{IJ}^{-1}(p_{IJ}(h_K))
\end{equation}
is a closed subset of $H$ such that
$\Phi_K\cap \bigoplus_{i\in K} G_i=\emptyset$.
Since $H$ is compact, so is $\Phi_K$.

Observe that $\mathscr{N}_K=\{N_S: S\subseteq I\setminus K$ is finite$\}$
is a family of closed subsets of $H$ having the finite intersection property.
Clearly,
$\bigcap \mathscr{N}_K\subseteq\bigoplus_{i\in K} G_i$,
which yields
$\Phi_K\cap \bigcap \mathscr{N}_K=\emptyset$. Since $\Phi_K$ is compact,
$\Phi_K\cap N_F=\emptyset$ for some finite set $F\subseteq I\setminus K$.

Let $h\in N_F$ be arbitrary. Then $h\not\in\Phi_K$, and since
$h\in H$, \eqref{def:Phi:K} implies that $p_{IJ}(h)\not=p_{IJ}(h_K)$.
This shows that
$p_{IJ}(h_K)\not\in p_{IJ}(N_F)$.
Since $p_{IJ}(h_K)\in p_{IJ}(H)$, we get
$p_{IJ}(h_K)\in p_{IJ}(H)\setminus p_{IJ}(N_F)\not=\emptyset$.
\end{proof}

Let $[I]^{<\omega}$ be the family of all finite subsets of $I$. Our claim allows us to define a map $\sigma: [I]^{<\omega}\to[I]^{<\omega}$ such that
\begin{equation}
\label{eq:5:new}
K\cap \sigma(K)=\emptyset
\mbox{ and }
p_{IJ}(H)\setminus p_{IJ}(N_{\sigma(K)})\not=\emptyset
\mbox{ for all }
K\in [I]^{<\omega}.
\end{equation}

Let $S_0=J$ and let $S_{n+1}=S_n\cup \sigma(S_n)$ for every $n\in\N$.
The set $S=\bigcup\{S_n:n\in\N\}$ is at most countable.
Clearly, $J=S_0\subseteq S$.

\begin{claim}
\label{projection:not:uniformly:controllable}
The subgroup $H_S=p_{IS}(H)$ of $G_S=p_{IS}(G)=\prod_{i\in S} G_i$
is not uniformly controllable in $G_S$.
\end{claim}
\begin{proof}
Since $J$ is a finite subset of $S$,
it suffices
to show that
\begin{equation}
\label{eq:5:5}
p_{SJ}(H_S)\setminus p_{SJ}\left(H_S\cap \bigoplus_{i\in K} G_i\right)\not=\emptyset
\end{equation}
for every finite set
$K\subseteq S$.
Fix such a $K$.
Since the sequence $\{S_n:n\in\N\}$ is monotonically increasing
and
$S=\bigcup\{S_n:n\in\N\}$,
there exists
$n\in\N$ with $K\subseteq S_n$.

We claim that
\begin{equation}
\label{eq:Sn}
p_{SJ}\left(H_S\cap \bigoplus_{i\in K} G_i\right)\subseteq p_{IJ}(N_{\sigma(S_n)}).
\end{equation}
Indeed,
let $g\in H_S\cap \bigoplus_{i\in K} G_i$ be arbitrary.
Since $g\in H_S=p_{IS}(H)$, we can
choose $h\in H$ such that $g=p_{IS}(h)$;
that is, $h\restriction_S=g$.
Since
$K\subseteq S_n$, $\sigma(S_n)\cap S_n=\emptyset$
and
$g\in\bigoplus_{i\in K} G_i$, we conclude that
$
g(i)=1
$
for all $i\in \sigma(S_n)$.
Since
$\sigma(S_n)\subseteq S_{n+1}\subseteq S$,
it follows that
$h(i)=g(i)=1$ for all $i\in \sigma(S_n)$.
Hence,
$h\in N_{\sigma(S_n)}$
by \eqref{def:of:NF}.
Therefore,
$p_{SJ}(g)=p_{SJ}(p_{IS}(h))=p_{IJ}(h)\in p_{IJ}(N_{\sigma(S_n)})$.
This finishes the proof of
\eqref{eq:Sn}.

From \eqref{eq:5:new} and \eqref{eq:Sn}, we obtain
$$
\emptyset\not=p_{IJ}(H)\setminus p_{IJ}(N_{\sigma(S_n)})
\subseteq
p_{IJ}(H)\setminus p_{SJ}\left(H_S\cap \bigoplus_{i\in K} G_i\right).
$$
Since $p_{IJ}(H)=p_{SJ}(p_{IS}(H))=p_{SJ}(H_S)$,
this gives
\eqref{eq:5:5}.
\end{proof}

As a continuous image of the compact group $H$, the subgroup
$H_S$ of $G_S=\prod_{i\in S} G_i$ is compact.
Since $S$ is countable,  it follows from
Lemma \ref{compact:in:countable:products}
and Claim \ref{projection:not:uniformly:controllable} that
$H_S$ is not controllable in $G_S$.
Combining this with Proposition \ref{preservation:by:projections}(i), we obtain
that $H$ is not controllable in $G$.
\end{proof}

\begin{corollary}
\label{controllable:and:uniformly:controllable:coincide:in:closed:subgroups:of:compact:products}
Let $\{G_i:i\in I\}$ be a family of compact groups.
Then a closed subgroup $H$ of $G=\prod_{i\in I} G_i$ is controllable in $G$ if and only if $H$ is uniformly controllable in $G$.
\end{corollary}

\section{Controllability vs uniform controllability}

\label{Controllable:non-uniformly:controllable}

In this section we characterize abelian groups $M$ such that
every controllable subgroup of the power $M^I$ is uniformly controllable;
see Theorem
\ref{characterization:of:powers:in:which:all:controllable:subgroups:are:uniformaly:controllable}.
First, we develop some machinery necessary for the proof of this theorem.

Let $M$ be
an abelian
group,
and let $\mathscr{A}=\{A_i:i\in\N\}$ be a strictly ascending chain of subgroups of $M$; that is,
\begin{equation}
\label{chain:of:A_i}
A_0\subsetneq A_1\subsetneq A_2\subsetneq\dots\subsetneq A_i\subsetneq A_{i+1}\subsetneq\dots ,
\end{equation}
and each $A_i$ is a subgroup of $M$.
For each $i\in\N$,
\begin{equation}
\label{H_A:1}
H_{\mathscr{A},i}=\{(a,a,\dots,a)\in M^{i+1}: a\in A_i\}\times\prod_{j=i+1}^\infty \{0\}
\end{equation}
is a subgroup of $G=M^\N$.
Clearly,
\begin{equation}
\label{H_A:2}
H_{\mathscr{A}}=\langle \bigcup_{i\in\N} H_{\mathscr{A},i}\rangle
\end{equation}
is a subgroup of $G$ associated with the chain $\mathscr{A}$.
Note that every element
$h\in H_{\mathscr{A}}$
has a representation
\begin{equation}
\label{decomposition:of:h}
h=\sum_{i=0}^m h_i,
\mbox{ where }
h_i\in H_{\mathscr{A}, i}
\mbox{ for all }
i=0,\dots,m
\mbox{ and }
h_m\not=0.
\end{equation}

We claim that
\begin{equation}
\label{intersection:of:H_A:with:finite:faces}
H_{\mathscr{A}}\cap \bigoplus_{i=0}^k M=\sum_{i=0}^k H_{\mathscr{A},i}
\hskip50pt
\mbox { for all }
k\in\N.
\end{equation}
Fix $k\in\N$.
Let $h$ be an element of the set on the left hand side of \eqref{intersection:of:H_A:with:finite:faces}. Then $h$ has a representation
as in \eqref{decomposition:of:h}.
Note that
$h_m\in H_{\mathscr{A},m}$, \eqref{H_A:1} and $h_m\not=0$ yield
$h_m(m)\not=0$.
Now \eqref{H_A:1} and \eqref{decomposition:of:h} give
$
h(m)=\sum_{i=0}^m h_i(m)=h_m(m)\not=0.
$
Since $h\in \bigoplus_{i=0}^k M$, from this we conclude that $m\le k$.
Combining this with \eqref{decomposition:of:h}, we obtain
$h\in \sum_{i=0}^m H_{\mathscr{A},i}\subseteq \sum_{i=0}^k H_{\mathscr{A},i}$.
This argument proves that
the set on the left hand side of \eqref{intersection:of:H_A:with:finite:faces}
is a subset of the set on the right hand side of the same equation.
The converse inclusion follows easily
from \eqref{H_A:1} and \eqref{H_A:2}.

\begin{claim}
\label{H_A:controllable:but:not:uniformly}
\begin{itemize}
\item[(i)]
$H_{\mathscr{A}}$ is controllable in $G$ but not uniformly controllable in $G$.
\item[(ii)]
If all $A_i$ ($i\in\N$) are countable, then $H_{\mathscr{A}}$ is countable as well.
\end{itemize}

\end{claim}
\begin{proof}
(i)
Since $H_{\mathscr{A}}\subseteq \bigoplus_{k\in\N} M$, the subgroup $H_{\mathscr{A}}$ of $G$ is controllable in $G$ by Proposition \ref{examples:of:controllability}(ii).

Let $J=\{0\}$.
Let $K$ be an arbitrary finite subset of $\N$.
Then $K\subseteq \{0,1,\dots,k\}$ for some $k\in\N$.
Combining this with \eqref{chain:of:A_i},
\eqref{H_A:1},
\eqref{intersection:of:H_A:with:finite:faces}
and recalling the fact that $p_{\N J}$ is a homomorphism,
we get
$$
p_{\N J}\left(H_{\mathscr{A}}\cap \bigoplus_{i\in K} M \right)
\subseteq
p_{\N J}\left(H_{\mathscr{A}}\cap \bigoplus_{i=0}^k M\right)
=
p_{\N J}\left(\sum_{i=0}^k H_{\mathscr{A},i}\right)
=
\sum_{i=0}^k p_{\N J}(H_{\mathscr{A},i})
=
\sum_{i=0}^k A_i
=A_k,
$$
where the last equation follows from the fact that $A_k$ is a subgroup of $M$.
Similarly,
since all $A_i$ ($i\in\N$) are subgroups of $M$
and
$p_{\N J}$ is a homomorphism,
from
\eqref{chain:of:A_i},
\eqref{H_A:1} and
\eqref{H_A:2}, we obtain
$$
p_{\N J}(H_{\mathscr{A}})=\langle \bigcup_{i\in\N} p_{\N J}(H_{\mathscr{A},i})\rangle
=
\langle \bigcup_{i\in\N} A_i\rangle
=
\bigcup_{i\in\N} A_i.
$$
Since $A_k$ is a proper subset of $\bigcup_{i\in\N} A_i$ by
\eqref{chain:of:A_i},
we deduce that
$$
p_{\N J}(H_{\mathscr{A}})\not=p_{\N J}\left(H_{\mathscr{A}}\cap \bigoplus_{i\in K} M\right).
$$
This shows that
$H_{\mathscr{A}}$ is not uniformly controllable in $G$.

(ii)
This easily follows from \eqref{H_A:1} and \eqref{H_A:2}.
\end{proof}

\begin{theorem}
\label{characterization:of:powers:in:which:all:controllable:subgroups:are:uniformaly:controllable}
For an abelian group $M$, a set $I$ and the group $G=M^I$, the following conditions are equivalent:
\begin{itemize}
\item[(i)] every countable controllable subgroup of $G$ is uniformly controllable in $G$;
\item[(ii)] every controllable subgroup of $G$ is uniformly controllable in $G$;
\item[(iii)] either $I$ is finite or $M$ is finitely generated.
\end{itemize}
\end{theorem}
\begin{proof}
(i)$\to$(iii)
We assume that (iii) fails, and we shall prove that (i) fails as well.
Since (iii) fails, the set $I$ is infinite and the group $M$ is not finitely generated.
The latter assumption allows us to
find a chain $\mathscr{A}=\{A_k:k\in\N\}$ of finitely generated subgroups $A_k$ of $M$ satisfying
\eqref{chain:of:A_i}.
By Claim
\ref{H_A:controllable:but:not:uniformly},
the corresponding subgroup
$H_{\mathscr{A}}$
of $M^\N$ is controllable in $M^\N$ but not uniformly controllable in
$M^\N$.
Since $I$ is infinite, we can
fix a countably infinite subset $N$ of
$I$.
Let $\varphi:N\to \N$ be a bijection.
It naturally induces the group isomorphism
$\theta: M^\N\to M^N$ defined by $\theta(h)=h\circ \varphi$ for
$h\in M^\N$.
One easily checks that
$H'=\theta(H_{\mathscr{A}})$
is a controllable subgroup of $M^N$ that is not uniformly controllable in $M^N$.
Now $H=H'\times \{\mathbf{0}\}$
is a controllable subgroup of $G$
that is not uniformly controllable in $G$.
(Here $\mathbf{0}$ is the zero element of the group $M^{I\setminus N}$.)
Therefore, (i) fails.

(iii)$\to$(ii)
If the set $I$ is finite, then (ii) trivially holds, as one can take $K=I$ in the definition of uniform controllability.
If the group $M$ is finitely generated, then it satisfies the ascending chain condition. Applying Proposition \ref{all:groups:satisfy:the:ascending:chain:condition}, we conclude that (ii) holds.

(ii)$\to$(i) is trivial.
\end{proof}

The next example demonstrates that compactness of $H$ is essential in Theorem
\ref{controllable:and:uniformly:controllable:coincide:for:compact:subgroups}
and closedness of $H$ is essential in Corollary
\ref{controllable:and:uniformly:controllable:coincide:in:closed:subgroups:of:compact:products}.

\begin{example}
\label{power:of:Z(2)}
Let $M=\Z(2)^\N$ be the infinite compact metric group of order $2$.
Since $M$ is not finitely generated, Theorem
\ref{characterization:of:powers:in:which:all:controllable:subgroups:are:uniformaly:controllable}
implies that  {\em $G=M^\N$ has a countable subgroup $H$
which is controllable in $G$ but not uniformly controllable in $G$\/}. Clearly, {\em $H$ has order $2$\/} as well.
\end{example}

Item (i) of
our next theorem demonstrates that whenever the countable power $G=M^\N$ of some abelian group $M$ has at least one subgroup that is controllable in $G$ but is not uniformly controllable in $G$, then the family of all such subgroups is rather large. Item (ii) of this theorem establishes the version of this fact for countable subgroups of $G$.

\begin{theorem}
\label{many:controllable:non-uniformly:controllable}
Let $M$ be an abelian group,
$\mathscr{H}$ be the family of all subgroups of $G=M^\N$ which are controllable in $G$ but not uniformly controllable in $G$,
and let $\mathscr{H}_c=\{H\in\mathscr{H}: |H|\le|\N|\}$.
\begin{itemize}
\item[(i)]
Either $\mathscr{H}=\emptyset$ or
$|\mathscr{H}|\ge 2^{|M|}\ge \cont$.
\item[(ii)]
If $M$ is countably infinite, then either $\mathscr{H}_c=\emptyset$ or
$|\mathscr{H}_c|=\cont$.
\end{itemize}
\end{theorem}
\begin{proof}
Suppose that $\mathscr{H}\not=\emptyset$. (In particular, this assumption trivially holds when $\mathscr{H}_c\not=\emptyset$.)
Applying
Theorem \ref{characterization:of:powers:in:which:all:controllable:subgroups:are:uniformaly:controllable},
we conclude that
$M$
is not finitely generated.
In particular, $M$ is infinite, and so $2^{|M|}\ge \cont$.
We consider two cases, depending on the cardinality of the group $M$.

\smallskip
{\em Case 1.\/} {\sl $M$ is countably infinite.\/}
In this case, $2^{|M|}=\cont$.
To prove both items (i) and (ii),
it suffices to show that
$|\mathscr{H}_c|\ge \cont$.
Indeed, $|\mathscr{H}|\ge |\mathscr{H}_c|$, which yields (i).
Since $|M^\N|=|\N|^{|\N|}=\cont$, we have
$|\mathscr{H}_c|\le |M^\N|^\omega=\cont^\omega=\cont$.
Combining it with $|\mathscr{H}_c|\ge \cont$, we get (ii).

Since $M$ is not finitely generated, by
Theorem \ref{characterization:of:powers:in:which:all:controllable:subgroups:are:uniformaly:controllable}
the product $M^\N$ contains a countable subgroup $K$ that is controllable in $M^\N$ but not uniformly controllable in $M^\N$.
Now we proceed similarly to the proof of the implication (i)$\to$(iii) of
Theorem
\ref{characterization:of:powers:in:which:all:controllable:subgroups:are:uniformaly:controllable}.
The family $\mathscr{N}$ of all infinite subsets of $\N$ has cardinality $\cont$. For every $N\in\mathscr{N}$ we fix a bijection $\varphi_N:N\to \N$, which induces
the group isomorphism
$\theta_N: M^\N\to M^N$ defined by $\theta_N(h)=h\circ \varphi_N$ for
$h\in M^\N$.
One easily checks that
$H_{N}'=\theta_N(K)$
is a countable controllable subgroup of $M^N$ that is not uniformly controllable in $M^N$.
Now $H_N=H_N'\times \{\mathbf{0}_N\}$
is a countable controllable subgroup of $G$
that is not uniformly controllable in $G$.
(Here $\mathbf{0}_N$ is the zero element of the group $M^{\N\setminus N}$.)
Therefore,
$\{H_N:N\in\mathscr{N}\}\subseteq \mathscr{H}_c$.
Finally, one easily sees that $H_{N_1}\not= H_{N_2}$ whenever $N_1,N_2\in\mathscr{N}$ and $N_1\not= N_2$.
This proves that $|\mathscr{H}_c|\ge |\{H_N:N\in\mathscr{N}\}|=|\mathscr{N}|=\cont$.

\smallskip
{\em Case 2.\/} {\sl $M$ is uncountable.\/}
In this case, $|M|=\sup\{r_p(M): p\in\mathbb{P}\cup\{0\}\}$,
where $\mathbb{P}$ is the set of all prime numbers and
$r_p(M)$ is a $p$-rank of $M$.
For every $p\in \mathbb{P}\cup\{0\}$ choose a $p$-free subset $Y_p$
of $M$ with $|Y_p|=r_p(M)$. Define $Y=\bigcup\{Y_p: p\in\mathbb{P}\cup\{0\}\}$.
Then $|Y|=\sup\{r_p(M): p\in\mathbb{P}\cup\{0\}\}=|M|$
and the set $Y$ is {\em independent\/};
that is, $\langle Z\rangle \cap \langle Y\setminus Z\rangle =\{0\}$ for every
non-empty set $Z\subseteq Y$.

Define $\mathscr{Z}=\{Z\subseteq Y: Z$ is infinite$\}$.
Let $Z\in\mathscr{Z}$. Since $Z$ is infinite,
we can choose a strictly increasing sequence $\{Z_i:i\in\N\}$ of subsets of $Z$
such that $Z=\bigcup_{i\in\N} Z_i$. Since $Z\subseteq Y$ and $Y$ is independent,
$\mathscr{A}(Z)=\{\langle Z_i\rangle: i\in\N\}$ is a strictly ascending chain of subgroups of $M$ with $\bigcup\mathscr{A}(Z)=\langle Z\rangle$.
By Claim
\ref{H_A:controllable:but:not:uniformly},
$H_{\mathscr{A}(Z)}\in \mathscr{H}$.
Therefore,
$|\mathscr{H}|\ge |\{H_{\mathscr{A}(Z)}:Z\in\mathscr{Z}\}|$.
Since $|\mathscr{Z}|=2^{|Y|}=2^{|M|}$, it remains only check that
$H_{\mathscr{A}(Z_0)}\not=H_{\mathscr{A}(Z_1)}$ whenever $Z_0,Z_1\in\mathscr{Z}$
and $Z_0\not= Z_1$.
Indeed, there exists $i=0,1$ such that $Z_i\setminus Z_{1-i}\not=\emptyset$.
Since $Z_0\cup Z_1\subseteq Y$ and $Y$ is independent,
it follows that
$Z_i\setminus \langle Z_{1-i}\rangle\not=\emptyset$.
Let $J=\{0\}$.
Since $p_{\N J}(H_{\mathscr{A}(Z_j)})=\langle Z_j\rangle$ for $j=0,1$,
we conclude that
$p_{\N J}(H_{\mathscr{A}(Z_i)})\setminus
p_{\N J}(H_{\mathscr{A}(Z_{1-i})})\not=\emptyset$.
This implies that $H_{\mathscr{A}(Z_0)}\not=H_{\mathscr{A}(Z_1)}$.
\end{proof}

\begin{corollary}
\label{lots:of:controllable:not:uniformly:controllable:subgroups}
Let $\mathscr{H}$ be the family of all subgroups of $\T^\N$ that are controllable in $\T^\N$ but not uniformly controllable in $\T^\N$.
Then $|\mathscr{H}|=2^\cont$.
\end{corollary}
\begin{proof}
Since $\T$ is not finitely generated, $\mathscr{H}\not=\emptyset$ by
Theorem
\ref{characterization:of:powers:in:which:all:controllable:subgroups:are:uniformaly:controllable}.
Therefore,
$|\mathscr{H}|\ge 2^{|\T|}=2^\cont$
by Theorem \ref{many:controllable:non-uniformly:controllable}.
Since $|\T^\N|=\cont$, the inverse inclusion
$|\mathscr{H}|\le 2^{|\T^\N|}=2^\cont$ also holds.
\end{proof}

\section{Weak controllability vs controllability}

\label{weakly:controllable:non:controllable}

In this section we demonstrate
that the notions or weak controllability and controllability
differ for subgroups
of $\T^\N$, thereby
showing that the assumption that each $G_i$ is discrete
in Proposition \ref{controllability}
cannot be replaced with the assumption that each $G_i$ is
compact,
and
the assumption that each $G_i$ is finite in
Corollary
\ref{uniform:controllability}
cannot be weakened to the assumption that each $G_i$ is
compact,
even in the abelian case.
In fact, we prove that the family $\mathscr{H}$ of all weakly controllable
non-controllable subgroups of $\T^\N$ has cardinality $2^\cont$, which is the maximal
size possible.
Furthermore, we exhibit a compact member of $\mathscr{H}$ in
Example \ref{compact:subgroup}
and
a countable torsion member of $\mathscr{H}$ in Example \ref{torsion:subgroup}.

First, we develop some machinery necessary for proving these results.

Let $Y=\{y_k:k\in\N\}\subseteq t(\T)\setminus\{0\}$ be a sequence
converging to $0$ in $\T$.
For every $k\in\N$
define $f_k\in \T^\N$ by letting
\begin{equation}
\label{eq:xn}
f_k(n)=
\left\{ \begin{array}{ll}
y_k & \mbox{ if } n\le k \\
0, & \mbox{ if } n>k
\end{array} \right.
\hskip50pt
\mbox{for }
n\in\N.
\end{equation}
Since $Y\subseteq t(\T)$,
for every $k\in\N$ we can fix $s_k\in\N$ such that $s_k y_k=0$.
From this and \eqref{eq:xn}, it follows that $s_k f_k=0$ for every $k\in\N$.
Therefore,
\begin{equation}
\label{cyclic:group:yk}
\langle f_k\rangle=\{m f_k: m\in\N, 0\le m<s_k\}
\mbox{ for every }
k\in\N.
\end{equation}
In particular,
\begin{equation}
\label{FY:DY}
F_Y=\{f_k:k\in\N\}\subseteq t(\T^\N).
\end{equation}
Let $D_Y=\langle F_Y\rangle$.

\begin{claim}
\label{claim:about:g}
If $g\in \overline{D_Y}$, then
$g(n)-g(n+1)\in \langle y_n \rangle$
for every $n\in\N$.
\end{claim}
\begin{proof}
Let
$g\in \overline{D_Y}$ and $n\in \N$.
Since $g\in \overline{D_Y}\subseteq \T^\N$
and $\T^\N$ is metric, we can fix a sequence
$\{h_j:j\in\N\}\subseteq D_Y$
converging to $g$.
For each $j\in\N$,
from $h_j\in D_Y=\langle F_Y\rangle=\langle \{f_k:k\in\N\}\rangle$
we conclude that
there exists a
representation
$h_j=\sum_{k\in\N} m_{j,k} f_k$ such that $\{m_{j,k}:k\in\N\}\subseteq \N$
and
the set $\{k\in\N:m_{j,k}\not=0\}$ is finite; moreover,
by
\eqref{cyclic:group:yk},
we may assume, without loss of generality, that
$0\le m_{j,k}< s_k$
for all $k\in\N$.
Since $m_{j,n}\in \{0,1,\dots,s_n-1\}$
for all $j\in\N$, there exist
$m\in \N$
and
an infinite
set $J\subseteq\N$ such that
$m_{j,n}=m$
for all $j\in J$.
Since the sequence $\{h_j:j\in\N\}$ converges to $g$, its subsequence
$\{h_j:j\in J\}$ converges to $g$ as well.

Let $j\in J$ be arbitrary.
By \eqref{eq:xn},
$$
h_j(n)=\sum_{k\in\N} m_{j,k} f_k(n)=\sum_{k\le n-1} m_{j,k} f_k(n)
+m_{j,n} f_n(n)
+\sum_{k\ge n+1} m_{j,k} f_k(n)
=
m y_n + \sum_{k\ge n+1} m_{j,k} y_k
$$
and
$$
h_j(n+1)=\sum_{k\in\N} m_{j,k} f_k(n+1)=
\sum_{k\le n} m_{j,k} f_k(n+1)
+\sum_{k\ge n+1} m_{j,k} f_k(n+1)
=
\sum_{k\ge n+1} m_{j,k} y_k,
$$
so
$h_j(n)-h_j(n+1)=my_n\in \langle y_n\rangle$.
Since $\langle y_n\rangle$ is a finite subgroup of $\T$, it is closed
in $\T$, and so
$$
g(n)-g(n+1)=
\lim_{j\to\infty, j\in J}
h_j(n)
-
\lim_{j\to\infty, j\in J}
h_j(n+1)
=
\lim_{j\to\infty, j\in J}
(h_j(n)-h_j(n+1))
\in \langle y_n\rangle,
$$
as required.
\end{proof}

For every $x\in \T$, define $c_x\in\T^\N$ by letting
 $c_x(n)=x$ for all $n\in \N$.

\begin{claim}
\label{weakly:controllable:not:controllable}
Suppose that $x\in\T\setminus \langle Y\rangle$
and $H$ is a subgroup of $\T^\N$ such that
$D_Y\subseteq H\subseteq \overline{D_Y}$
and $c_x\in H$.
Then $H$ is weakly controllable but not controllable.
\end{claim}
\begin{proof}
Note that $F_Y\subseteq \bigoplus_{n\in\N}\T$ by
\eqref{eq:xn}
and
\eqref{FY:DY},
so $D_Y=\langle F_Y\rangle \subseteq \bigoplus_{n\in\N}\T$
as well. Since $D_Y\subseteq H$,
we get
$D_Y\subseteq H\cap\bigoplus_{n\in\N}\T$.
Since $D_Y$ is dense in $H$,
we conclude that
$H\cap\bigoplus_{n\in\N}\T$ is dense in $H$ as well.
This shows that $H$ is weakly controllable.

Suppose that $H$ is controllable.
Applying the definition of controllability to $J=\{0\}$,
we get
$x=p_{\N J}(c_x)\in p_{\N J}(H)=p_{\N J}(H\cap\bigoplus_{i\in\N} \T)$,
so there exists $g\in H\cap\bigoplus_{i\in\N} \T$ such that
$x=p_{\N J}(g)=g(0)$.
Since $g\in \bigoplus_{i\in\N} \T$,
there exists $n\in \N$ such that $g(i)=0$ for all integers $i>n$.
In particular,
\begin{equation}
\label{eq:x}
x=g(0)=g(0)-g(n+1).
\end{equation}
Since $g\in H\subseteq \overline{D_Y}$,
from Claim \ref{claim:about:g}
we conclude that
$g(i)-g(i+1)\in \langle y_i\rangle$
for all $i=0,\dots,n$.
Therefore,
$$
g(0)-g(n+1)
=
\sum_{i=0}^n
(g(i)-g(i+1))
=
\langle y_0\rangle
+
\langle y_1\rangle
+
\dots
+
\langle y_n\rangle
\subseteq  \langle Y\rangle.
$$
Combining this with \eqref{eq:x},
we get $x\in \langle Y\rangle$,
in contradiction with our assumption on $x$.
This contradiction shows that $H$ is not controllable.
\end{proof}

\begin{claim}
\label{constant:in:the:closure}
$c_x\in \overline{D_Y}$ for all $x\in\T$.
\end{claim}
\begin{proof}
Fix $x\in \T$.
Let $O$ be an open neighbourhood of $c_x$ in $\T^\N$.
There exist $n\in\N$ and an open
neighbourhood $U$ of $x$ in $\T$ such that
\begin{equation}
\label{VV}
V=\{f\in\T^\N: f(i)\in U
\mbox{ for all }
i=0,\dots,n\}
\subseteq O.
\end{equation}
Since $\lim_{k\to\infty} y_k=0$, we can choose
$k\in\N$ such that $k\ge n$ and $\langle y_k\rangle\cap U\not=\emptyset$.
 Fix
$m\in\mathbb{Z}$
such that $m y_k\in U$.
By \eqref{eq:xn},
$m f_k(i)=m y_k\in U$ for every $i=0,\dots,n$.
Therefore,
$m f_k\in V\subseteq O$
by \eqref{VV}.
Since $m f_k\in D_Y$,
we have $O\cap D_Y\not=\emptyset$.
This shows that $c_x\in\overline{D_Y}$.
\end{proof}

\begin{claim}
\label{generating:the:subgropup}
If $X\subseteq \T$ and $X\setminus \langle Y \rangle\not=\emptyset$,
then $H=\langle F_Y\cup\{c_x:x\in X\}\rangle$ is a weakly controllable subgroup of
$\T^\N$ that is not controllable.
\end{claim}
\begin{proof}
Clearly, $D_Y=\langle F_Y\rangle\subseteq H$.
Let us check the inclusion $H\subseteq \overline{D_Y}$.
Obviously, $F_Y\subseteq D_Y\subseteq \overline{D_Y}$.
Furthermore,
$\{c_x:x\in X\}\subseteq \overline{D_Y}$ by Claim \ref{constant:in:the:closure}.
Since $D_Y$ is a subgroup of $\T^\N$, so is its closure $\overline{D_Y}$.
Therefore, $H=\langle F_Y\cup\{c_x:x\in X\}\rangle\subseteq \overline{D_Y}$.
Since $X\setminus \langle Y \rangle\not=\emptyset$, there exists
$x\in X\setminus \langle Y \rangle$. Since $c_x\in H$,
Claim \ref{weakly:controllable:not:controllable}
can be applied.
\end{proof}

\begin{example}
\label{compact:subgroup}
Let $Y=\{y_k:k\in\N\}\subseteq t(\T)$ be any sequence
converging to $0$ in $\T$.
Then {\em $H=\overline{D_Y}$ is a closed subgroup
 of $\T^\N$ which is weakly controllable but not controllable\/}.
Indeed, since the set $\langle Y\rangle$ is countable
and $\T$ is uncountable,
we can
fix $x\in \T\setminus \langle Y\rangle$.
By Claim \ref{constant:in:the:closure},
$c_x\in \overline{D_Y}=H$, and
Claim \ref{weakly:controllable:not:controllable} can be applied.
\end{example}

\begin{example}
\label{torsion:subgroup}
Let $\Prm\setminus\{2\}=\{p_k:k\in\N\}$ be a faithful enumeration of all
prime numbers other than $2$.
For every $k\in\N$ let $y_k\in\T$ be one of the two
elements of order $p_k$
which are closest to $0$ (in the natural metric of $\T$).
Then $Y=\{y_k:k\in\N\}\subseteq t(\T)$ is a sequence converging to $0$ in $\T$.
Let $x\in\T$ be the element of order $2$. Then
\emph{$H=\langle F_Y\cup \{c_x\}\rangle$ is a countable torsion subgroup of $\T^\N$ that is weakly controllable but not controllable\/}.
Clearly, $H$ is countable.
Since $x$ has order $2$, the constant function $c_x\in\T^\N$
also has order $2$; in particular, $c_x\in t(\T^\N)$.
Combining this with \eqref{FY:DY}, we conclude that
$H=\langle F_Y\cup \{c_x\}\rangle\subseteq t(\T^\N)$;
that is, $H$ is torsion.
It remains only to show that
$X=\{x\}$ and $H$ satisfy the assumption of
Claim \ref{generating:the:subgropup}.
Observe that every non-zero element of the group
$\langle Y\rangle=\bigoplus_{k\in\N} \langle y_k\rangle$ has order
$p_1 p_2 \dots p_i$ for suitable $p_1, p_2, \dots, p_i\in\Prm\setminus\{2\}$,
so
$\langle Y\rangle\setminus\{0\}$ has only elements of odd order.
Since $x$ has an even order $2$, we have $x\in\T\setminus \langle Y\rangle$.
Thus,
$x\in \{x\}\setminus \langle Y\rangle=X\setminus \langle Y\rangle\not=\emptyset$.
\end{example}

\begin{theorem}
\label{many:weakly:controllable:not:controllable;subgroups}
Let $\mathscr{H}$ be the family of subgroups $H$ of $\T^\N$ which are
weakly controllable but not controllable. Then $|\mathscr{H}|=2^\cont$.
\end{theorem}
\begin{proof}
Since $|\T^\N|=\cont$, we have $|\mathscr{H}|\le 2^\cont$.
It remains only to show that $|\mathscr{H}|\ge 2^\cont$.

Let
$Y=\{y_k:k\in\N\}\subseteq t(\T)\setminus\{0\}$ be any sequence
converging to $0$ in $\T$. Since $\T$ has rank $\cont$, we can
fix an independent set $Z\subseteq \T\setminus t(\T)$ with $|Z|=\cont$.
Then the family $\mathscr{X}=\{X\subseteq Z: X\not=\emptyset\}$
of all non-empty subsets of $Z$
has cardinality $2^\cont$.
Since $Z$ is an independent subset of $\T$, the set $C=\{c_x:x\in Z\}$ is an independent
subset of $\T^\N$.

For every $X\in\mathscr{X}$, define
$H_X=\langle F_Y\cup\{c_x:x\in X\}\rangle$.

\begin{claim}
\label{distinct}
If $X_0,X_1\in\mathscr{X}$ and $X_0\not= X_1$, then $H_{X_0}\not=H_{X_1}$.
\end{claim}
\begin{proof}
There exist $i\in\{0,1\}$ and $x^*\in X_i\setminus X_{1-i}$.
Clearly, $c_{x^*}\in H_{X_i}$, so it remains only to show that
$c_{x^*}\not\in H_{X_{1-i}}$. Suppose that
$c_{x^*}\in H_{X_{1-i}}=\langle F_Y\cup\{c_x:x\in X_{1-i}\}\rangle$.
Then $c_{x^*}=g+h$ for some $g\in \langle F_Y\rangle$ and
$h\in \langle\{c_x:x\in X_{1-i}\}\rangle$.
Note that $X_0\cup X_1\subseteq Z$, so
$g=c_{x^*}-h\in\langle C\rangle$.
Since $g\in \langle F_Y\rangle$, we have $g\in t(\T^\N)$ by
\eqref{FY:DY}.
Thus, $g\in \langle C\rangle\cap t(\T^\N)=\{0\}$ (as $C$ is independent).
This gives $g=0$ and $c_{x^*}=h\in\langle\{c_x:x\in X_{1-i}\}\rangle$.
Since $x^*\not\in X_{1-i}$, $c_{x^*}\not\in \{c_x:x\in X_{1-i}\}$.
This contradicts the fact that $C$ is independent.
\end{proof}

\begin{claim}
\label{weakly:contr:not:contr}
For every $X\in\mathscr{X}$, the subgroup $H_X$ of $\T^\N$ is weakly controllable but not controllable.
\end{claim}
\begin{proof}
Since $X\not=\emptyset$, we can fix $x\in X$.
Since $X\subseteq Z$ and $Z$ is independent in $\T$, $x$ has infinite order in $\T$.
Since $Y\subseteq t(\T)$, we have $\langle Y\rangle \subseteq t(\T)$.
This yields $x\in X\setminus \langle Y\rangle$. Now the conclusion follows from
Claim \ref{generating:the:subgropup}.
\end{proof}

From Claims \ref{distinct} and \ref{weakly:contr:not:contr}, we obtain
the inequality $|\mathscr{H}|\ge |\{H_X:X\in \mathscr{X}\}|=|\mathscr{X}|=2^\cont$.
\end{proof}

\section{Behavior of controllability under taking closures and dense subgroups}
\label{sec:4}

In this section we investigate how the three notions of controllability behave under the closure operation.
\begin{proposition}
Let $I$ be a set and $\{G_i:i\in I\}$ be a family of topological groups.
Let $H$ and $H'$ be subgroups of $G=\prod_{i\in I} G_i$ such that
$H\subseteq H'\subseteq \overline{H}$. If $H$ is weakly controllable in $G$, then so is $H'$.
\end{proposition}

\begin{corollary}
The closure of a weakly controllable subgroup is weakly controllable.
\end{corollary}

Our next example shows that the situation changes completely in the case of controllability.

\begin{example}
\label{the:closure:of:controllable:subgroup:need:not:be:controllable}
{\em The closure of a (countable torsion) controllable
subgroup of $\T^\N$
need not be controllable\/}.
Indeed,
let $H=\overline{D_Y}$ be the closed non-controllable subgroup of $\T^\N$ constructed in Example
\ref{compact:subgroup}.
Since
$D_Y\subseteq \bigoplus_{n\in\N} t(\T)\subseteq \bigoplus_{n\in\N} \T$,
it follows from Proposition \ref{examples:of:controllability}(ii)
that $D_Y$ is controllable.
\end{example}

Since the group $H$ in the above example is compact, this example shows that ``uniformly'' cannot be omitted either from our next theorem or from its Corollary \ref{closures:of:uniformally:controllable:subgroups:in:products:of:compact:groups} below.

\begin{theorem}
\label{the:compact:closure:of:uniformly:controllable:subgroup:is:uniformly:controllable}
Let $I$ be a set and $\{G_i:i\in I\}$ be a family of topological groups.
Furthermore, let $H$ be a subgroup of $G=\prod_{i\in I} G_i$ such that
$\overline{H}$ is compact. If $H$ is uniformly controllable, then so is
$\overline{H}$.
\end{theorem}
\begin{proof}
Let $J$ be a finite subset of $I$. Since $H$ is uniformly controllable,
there exists a finite set $K\subseteq I$ such that
\begin{equation}
\label{eq:image:of:finite:sum}
p_{IJ}(H)=p_{IJ}\left(H\cap \bigoplus_{i\in K} G_i\right)
\subseteq
p_{IJ}\left(\overline{H}\cap \bigoplus_{i\in K} G_i\right).
\end{equation}
Since $K$ is finite,
$\bigoplus_{i\in K} G_i$ is a closed subset of $G$, and so
$\overline{H}\cap \bigoplus_{i\in K} G_i$ is a closed subset of $\overline{H}$.
Since the latter set is compact, so is the former.
As a continuous image of the compact set $\overline{H}\cap \bigoplus_{i\in K} G_i$, the last set in \eqref{eq:image:of:finite:sum} is compact, and so it is
closed in
$\prod_{i\in J} G_i$.
Combining this with the inclusion from \eqref{eq:image:of:finite:sum},
we get
$\overline{p_{IJ}(H)}\subseteq p_{IJ}\left(\overline{H}\cap \bigoplus_{i\in K} G_i\right)$.
Since the map $p_{IJ}$ is continuous, we also have
$p_{IJ}(\overline{H})\subseteq \overline{p_{IJ}(H)}$.
This yields
$p_{IJ}(\overline{H})\subseteq p_{IJ}\left(\overline{H}\cap \bigoplus_{i\in K} G_i\right)$.
The inverse inclusion is obvious.
This finishes the proof of uniform controllability of $\overline{H}$ in $G$.
\end{proof}

\begin{corollary}
\label{closures:of:uniformally:controllable:subgroups:in:products:of:compact:groups}
Let $\{G_i:i\in I\}$ be a family of compact groups and $H$ be a uniformly
controllable subgroup of $G=\prod_{i\in I} G_i$. Then $\overline{H}$
is also uniformly controllable in $G$.
\end{corollary}

\begin{remark}
\label{controllable:not:uniformly:controllable}
Let $D_Y$ be the controllable subgroup of $\T^\N$ mentioned in
Example \ref{the:closure:of:controllable:subgroup:need:not:be:controllable}.
Since its closure $H=\overline{D_Y}$ in $\T^\N$ is not controllable,
it is not uniformly controllable either. Therefore, $D_Y$ is not uniformly controllable by
Corollary \ref{closures:of:uniformally:controllable:subgroups:in:products:of:compact:groups}.
\end{remark}

Items (ii) and (iii) of our next example show that
a dense subgroup of a uniformly controllable group need not be even weakly controllable, which shows that controllability is badly destroyed by passing to a dense subgroup.

\begin{example}
Let $K$ be a non-trivial abelian topological group.

\smallskip
(i) {\em There exists a dense subgroup $H$ of $K^\N$
such that $H\cap \bigoplus_{n\in\N} K=\{0\}$ and
$|H|\le
\max\{w(K),\omega\}$, where $w(K)$ is the weight of $K$.\/}
Indeed, let $D$ be a dense subgroup of $K$ such that $|D|=w(K)$.
Fix a partition $\N=\bigcup_{k\in\N} I_k$ of $\N$ into pairwise disjoint infinite
sets
$I_k$. For each $k\in\N$ define
\begin{equation}
\label{S_k}
S_k=I_k \setminus\{0,1,\dots,k\}.
\end{equation}
For $d\in D$ and $k\in\N$ we define $x_{d,k}\in K^\N$ by letting
\begin{equation}
\label{def:x}
x_{d,k}(n)=
\left\{ \begin{array}{ll}
d, & \mbox{ if } n \in S_k\cup\{k\} \\
0, & \mbox{ if } n\in \N\setminus(S_k\cup\{k\})
\end{array} \right.
\hskip50pt
\mbox{for }
n\in\N.
\end{equation}
Since $D$ is a subgroup of $K$, it follows from
\eqref{def:x} that
$X_k=\{x_{d,k}:d\in D\}$ is a subgroup of $K^\N$ for every $k\in\N$.
We claim that
\begin{equation}
\label{def:H}
H=\langle \bigcup_{k\in\N} X_k\rangle
\end{equation}
is the desired subgroup of $K^\N$.
Clearly,
$|X_k|=|D|$ for every $k\in \N$, so
$
|H|\le
\max\{|D|,\omega\}=\max\{w(K),\omega\}.
$

Next, let us check
that
$H\cap \bigoplus _{n\in\N} K=\{0\}$.
Choose $h\in H\setminus\{0\}$ arbitrarily.
Since each $X_k$ is a subgroup of $K^\N$,
it easily follows from \eqref{def:H} that
$h=\sum_{i=0}^j x_{d_i, k_i}$
for suitable
$d_i$ and $k_i$ ($i=0,\dots,j$)
such that
$k_0,k_1,\dots,k_j$ are pairwise distinct.
Since $h\not=0$, there exists $l=0,\dots,j$ such that
$x_{d_l, k_l}\not=0$.
From this and \eqref{def:x} we conclude that
$d_l\not=0$.
Since $S_l$ is infinite, the set $S=S_l\setminus\{k_i: i=0,\dots, j\}$
is also infinite.

Let $n\in S$ be arbitrary.
Since $S_i\cap S_l=\emptyset$ for all $i=0,1,\dots, j$ with
$i\not=l$, it follows from
\eqref{def:x} and our definition of $S$ that
$x_{d_i, k_i}(n)=0$ whenever $i=0,1,\dots, j$ and $i\not=l$, so
$h(n)=x_{d_l, k_l}(n)=d_l\not=0$.
Since $S$ is infinite, this means that $h\not\in \bigoplus _{n\in\N} K$.
This finishes the proof of the equality
$H\cap \bigoplus _{n\in\N} K=\{0\}$.

It remains only to show that $H$ is dense in $K^\N$.
Fix an arbitrary non-empty open subset $W$ of $K^\N$. Then there exist $k\in\N$ and non-empty open subsets $U_0,\dots, U_k$ of $K$ such that
\begin{equation}
\label{basis:subproduct}
\left(\prod_{i=0}^k U_i\right)\times \prod_{j>k} K
\subseteq W.
\end{equation}
By (finite) induction on $j=0,1,\dots,k$, we select $d_j\in D$ as follows.
Since $D$ is dense in $K$,
we
fix
$d_0\in D\cap U_0$.
Using density of $D$ in $K$ again, we select
$d_1\in D\cap (U_1-x_{d_0,0}(1))$.
Assuming that $d_0,\dots,d_{j-1}$ were already defined, we use density of $D$ in $K$ again to choose
\begin{equation}
\label{d_j}
d_j\in D\cap \left(U_j-\sum_{i=0}^{j-1} x_{d_{i},i}(j)\right).
\end{equation}
Note that
\begin{equation}
\label{def:h}
h=\sum_{i=0}^k x_{d_i, i}\in H.
\end{equation}

Fix $j\in\{0,1,\dots,k\}$.
If $i\in\{j+1,\dots,k\}$, then
$x_{d_i,i}(j)=0$ by \eqref{S_k} and
\eqref{def:x}.
Combining this with \eqref{d_j} and \eqref{def:h},
we obtain
\begin{equation}
\label{h(j)}
h(j)=\sum_{i=0}^k x_{d_i, i}(j)
=
\sum_{i=0}^j x_{d_i, i}(j)
=
\left(\sum_{i=0}^{j-1} x_{d_i, i}(j)\right)
+
x_{d_j,j}(j)
=
\left(\sum_{i=0}^{j-1} x_{d_i, i}(j)\right)
+d_j
\in U_j.
\end{equation}
Since
\eqref{h(j)} holds for every
$j=0,1,\dots,k$,
from
\eqref{basis:subproduct}
we conclude that
$h\in W$.
From this and \eqref{def:h}, we get
$h\in W\cap H\not=\emptyset$.
Therefore, $H$ is dense in $K^\N$.

\smallskip
(ii) Let $H$ be the subgroup of $K^\N$ constructed in item (i).
Clearly, $H$ is not weakly controllable, while $K^\N$ is uniformly controllable
by Proposition \ref{examples:of:controllability}(i). This shows that {\em a dense subgroup of a uniformly controllable group need not be even weakly controllable.\/}

\smallskip
(iii) If one takes $K$ to be finite, then the subgroup $H$ of $K^\N$ constructed
in item (i) becomes countable. This allows us to conclude that
{\em a countable dense subgroup of a compact metric uniformly controllable group need not be even weakly controllable.\/}
\end{example}

\section{Connections with coding theory}
\label{coding:section}

Let $\Z$ be the set of integer numbers.
For $n\in\Z$, we let $n^-=\{k\in\Z: k\le n\}$ and $n^+=\{k\in\Z: k\ge n\}$.

Let us recall the classical notions from coding theory; see
\cite{trott:thesis,forney_trott:ie3trans93,fagnani:adv97,forney_trott:04}.
\begin{definition}
\label{classical:definition}
Assume that $I=\N$ or $I=\Z$.
Let
$\{G_i:i\in I\}$ be a family of topological groups
and $G=\prod_{i\in I} G_i$.
A subgroup $H$ of
$G$
is called:
\begin{itemize}
\item[(i)] {\em weakly controllable\/} in $G$
provided that $H\cap \bigoplus_{i\in I} G_i$ is dense in $H$;
\item[(ii)] {\em controllable\/}  in $G$ provided that for each pair
$h,h'$ of elements of $H$ and every integer $n\in I$ there exist an
integer $m\in I$ and an element $g\in H$ such that
$g\restriction_{n^-\cap I}=h\restriction_{n^-\cap I}$ and
$g\restriction_{m^+\cap I}=h'\restriction_{m^+\cap I}$;
\item[(iii)] {\em $k$-controllable\/} in $G$, for a fixed $k\in\N$,
provided that item (ii) holds with $m=n+k$;
\item[(iv)] {\em  strongly controllable\/}  in $G$
provided that $H$ is $k$-controllable for some $k\in\N$.
\end{itemize}
\end{definition}

Note that the topology of $G_i$, if any, is only used in item (i) and is irrelevant in
items (ii)--(iv).

These four notions play a prominent role in coding theory.
According to Forney and Trott \cite{forney_trott:04}, a \emph{group code} is a set of sequences
that has  a group property under a component-wise group operation. In this general setting,
a group code may also be seen as the behavior of a behavioral group system as given by Willens
\cite{willems:86,willems:97}.
It is known that many of the fundamental properties of linear codes and systems
depend only on their group structure. In fact, Forney and Trott, loc. cit., obtain purely
algebraic proofs of many of their results.
These notions are used in the study of
convolutional codes
that are well known and used currently in data transmission
(cf. \cite{forney_trott:ie3trans93}).

\begin{definition}
\label{uniform]controllability:in:Z}
Assume that $I=\N$ or $I=\Z$.
Let
$\{G_i:i\in I\}$ be a family of groups
and $G=\prod_{i\in I} G_i$.
We say that a subgroup $H$ of $G$ is
{\em uniformly controllable\/} in $G$
provided that for every integer $n\in I$
there exists an
integer $m\in I$
such that
for each pair
$h,h'$ of elements of $H$ one can find
an element $g\in H$ satisfying
$g\restriction_{n^-\cap I}=h\restriction_{n^-\cap I}$ and
$g\restriction_{m^+\cap I}=h'\restriction_{m^+\cap I}$.
\end{definition}

One can easily see that
\begin{equation}
\label{strongly:controllable:implications}
\mbox{strongly controllable}
\
\to
\
\mbox{uniformly controllable}
\
\to
\
\mbox{controllable}.
\end{equation}

Clearly, our notion of weak controllability in
Definition \ref{def:weak:controllability} is a direct
generalization
of the classical notion of weak controllability from Definition \ref{classical:definition}(i).
The next proposition shows that our notion of controllability in
Definition \ref{def:controllability}(i) is equivalent to the classical
notion of controllability from Definition \ref{classical:definition}(ii)
when
$I=\N$.

\begin{proposition}
\label{two:versions:of:controllability}
For every family
$\{G_i:i\in \N\}$ of groups,
a subgroup $H$ of
$G=\prod_{i\in \N} G_i$
is controllable in $G$ in the sense of
Definition \ref{def:controllability}(i)
if and only if $H$ is controllable in $G$ in the sense of
Definition \ref{classical:definition}(ii).
\end{proposition}
\begin{proof}
Assume that $H$ is controllable in $G$ in the sense of
Definition \ref{def:controllability}(i). To check that
$H$ is also controllable in $G$ in the sense of
Definition \ref{classical:definition}(ii),
fix $h,h'\in H$ and
$n\in\N$.
Since $H$ is a subgroup of $G$, we have $h^*=h(h')^{-1}\in H$.
Clearly,
the set
$J=n^-\cap \N$
is finite.
By Proposition \ref{reformulation:of:definition}(i),
there exists $g^*\in H\cap \bigoplus_{i\in \N} N$ such that
$g^*\restriction_J=h^*\restriction_J$.
Fix a finite set $K\subseteq \N$ with $g^*\in \bigoplus_{i\in K} G_i$.
Since $L=J\cup K$ is finite,
we can choose $m\in \N$ satisfying $m>\max L$.
Since $g^*,h'\in H$ and $H$ is a subgroup of $G$, one has
$g=g^*h'\in H$.
If $i\in n^-\cap \N=J$, then
$$
g(i)=g^*(i)h'(i)=h^*(i)h'(i)=h(i)(h')^{-1}(i) h'(i)=
h(i)h'(i)^{-1} h'(i)=h(i),
$$
which yields
$g\restriction_{n^-\cap \N}=h\restriction_{n^-\cap \N}$.
If $i\in m^+\cap \N$, then $i\not\in K$ by the choice of $m$, and so
$g^*(i)=1$ by the choice of $K$;
in particular,
$g(i)=g^*(i)h'(i)=h'(i)$.
This shows that
$g\restriction_{m^+\cap \N}=h'\restriction_{m^+\cap \N}$.

Suppose now that $H$ is controllable in $G$ in the sense of
Definition \ref{classical:definition}(ii).
To check that $H$ is controllable in $G$ in the sense of
Definition \ref{def:controllability}(i),
fix $h\in H$ and a finite set $J\subseteq \N$.
Choose
$n\in \N$ with $n\ge \max J$.
Since $H$ is controllable in $G$ in the sense of
Definition \ref{classical:definition}(ii), applying this definition
to $h$, $h'=1$ and $n$, we can find $m\in \N$ and $g\in H$
such that
$g\restriction_{n^-\cap \N}=h\restriction_{n^-\cap \N}$ and
$g\restriction_{m^+\cap \N}=h'\restriction_{m^+\cap \N}=1$.
Clearly,
the set $K=m^-\cap \N$ is finite.
From $g\restriction_{m^+\cap \N}=1$ and our definition of $K$
we conclude that
$g\in \bigoplus_{i\in K} G_i\subseteq \bigoplus_{i\in \N} G_i$.
Since $J\subseteq n^-\cap \N$, we have
$g(i)=g\restriction_{n^-\cap I}(i)=h(i)$ for every $i\in J$.
That is,
$g\restriction_J=h\restriction_J$.
Applying Proposition \ref{reformulation:of:definition}(i), we conclude that
$H$ is controllable in $G$ in the sense of Definition \ref{def:controllability}(i).
\end{proof}

Finally, it
is worth pointing out that the notion of uniform controllability in
Definition \ref{def:controllability}(ii) appears to be new even in the
case
when $I=\N$.
The next proposition shows that in this special case it coincides with
the notion from Definition \ref{uniform]controllability:in:Z}.

\begin{proposition}
\label{unif:contr:coincide}
For every family
$\{G_i:i\in \N\}$ of groups,
a subgroup $H$ of
$G=\prod_{i\in \N} G_i$
is uniformly controllable in $G$ in the
sense of
Definition \ref{def:controllability}(ii)
if and only if $H$ is uniformly controllable in $G$ in the sense of
Definition \ref{uniform]controllability:in:Z}.
\end{proposition}

The proof of this proposition is similar to that of Proposition
\ref{two:versions:of:controllability}, so we omit it.

The
main advantage of our Definition \ref{def:controllability} over the classical
Definition \ref{classical:definition}, as well as the adapted Definition \ref{uniform]controllability:in:Z},
is that
the index set $I$ is no longer required to be a subset of $\Z$;
in particular, the order structure on $I$ becomes irrelevant.

It is worth emphasizing,
based on Propositions \ref{two:versions:of:controllability}
and \ref{unif:contr:coincide},
that all our results in the first five sections
are
applicable to the classical case $I=\N$.
Furthermore, many of these results
are new even in this classical case.

A straightforward proof of the following proposition is left to the reader.

\begin{proposition}
\label{from:N:to:Z}
Let $M$ be a topological group and $H$ be a subgroup of $G=M^\N$.
Then $H'=M^{\Z\setminus\N}\times H$ is a subgroup of
$M^{\Z\setminus\N}\times M^\N=M^\Z=G'$ satisfying the following conditions:
\begin{itemize}
\item[(i)]
$H$ is weakly controllable in $G$ in the sense of Definition \ref{def:weak:controllability} if and only if $H'$ is weakly controllable in $G'$ in the sense of Definition \ref{classical:definition}(i).
\item[(ii)]
$H$ is controllable in $G$ in the sense of Definition \ref{def:controllability}(i) if and only if $H'$ is controllable in $G'$ in the sense of Definition \ref{classical:definition}(ii).
\item[(iii)]
$H$ is uniformly controllable in $G$ in the sense of Definition \ref{def:controllability}(ii) if and only if $H'$ is uniformly controllable in $G'$ in the sense of Definition \ref{uniform]controllability:in:Z}.
\item[(iv)]
$H$ is closed in $G$ if and only if $H'$ is closed in $G'$.
\end{itemize}
\end{proposition}

This proposition allows us to ``transform'' our examples distinguishing
properties in Definitions \ref{def:controllability} and \ref{def:weak:controllability}
into examples distinguishing
corresponding properties from the classical Definition \ref{classical:definition} and related Definition \ref{uniform]controllability:in:Z}.

Combining Theorem \ref{many:weakly:controllable:not:controllable;subgroups}
with items (i) and (ii) of
Proposition \ref{from:N:to:Z}, we obtain the following corollary.
\begin{corollary}
Let $\mathscr{H}$ be the family of subgroups $H$ of $G=\T^\Z$ which are
weakly controllable but not controllable in $G$ in the sense of Definition
\ref{classical:definition} (in which we let $I=\Z$).
Then $|\mathscr{H}|=2^\cont$.
\end{corollary}

Combining Example \ref{the:closure:of:controllable:subgroup:need:not:be:controllable}(i) with items (i), (ii) and (iv) of
Proposition \ref{from:N:to:Z}, we get the following corollary.

\begin{corollary}
There exists a closed (and thus, compact) subgroup $H$ of $G=\T^\Z$
which is weakly controllable but not controllable in $G$ in the sense of Definition
\ref{classical:definition} (in which we let $I=\Z$).
\end{corollary}

Combining Example \ref{controllable:not:uniformly:controllable}
with items (ii) and (iii) of
Proposition \ref{from:N:to:Z}, we obtain the following corollary.

\begin{corollary}
There exists a subgroup $H$ of $G=\T^\Z$
which is controllable in $G$ in the sense of Definition
\ref{classical:definition} but is not uniformly
controllable in $G$ in the sense of
Definition \ref{uniform]controllability:in:Z};
in particular, $H$ is not strongly controllable in $G$
in the sense of Definition \ref{classical:definition}.
 (Here we take $I=\Z$ in all definitions.)
\end{corollary}

\section{Uniform controllability vs strong controllability}

Our next theorem provides a general technique for building uniformly controllable
subgroups
that are not
strongly controllable.

\begin{theorem}
\label{distinguishing:uniform:and:k}
Let $I=\N$ or $I=\Z$.
Let $\{G_i:i\in I\}$ be a family of topological groups
and
let $G=\prod_{i\in I} G_i$ be its direct product with the Tychonoff
product topology.
Suppose that
$\{x_j:j\in\N\}\subseteq \left(\bigoplus_{i\in I} G_i\right)\setminus\{0\}$
and $H$ is a subgroup of $G$
satisfying two conditions:
\begin{itemize}
\item[(i)] $\{x_j:j\in\N\}\subseteq H\subseteq \overline{\langle \{x_j:j\in\N\}\rangle}$;
\item[(ii)]
$\max(\mathrm{supp} (x_s))< \min(\mathrm{supp} (x_t))$ whenever $s,t\in\N$ and $s<t$.
\end{itemize}
Then
$H$
is uniformly controllable in $G$.

Furthermore,
assume also that the following conditions are satisfied:
\begin{itemize}
\item[(iii)]
$\lim_{j\to\infty} |\mathrm{supp} (x_j)|=\infty$;
\item[(iv)]
if $j\in\N$ and $l,m\in\mathrm{supp} (x_j)$, then
$x_j(l)$ and $x_j(m)$ have the same order.
\end{itemize}
Then $H$ is not strongly controllable in $G$.
\end{theorem}
\begin{proof}
Let $K=\overline{\langle \{x_j:j\in\N\}\rangle}$.
It follows from (ii) that
$\mathrm{supp} (x_i)\cap \mathrm{supp} (x_j)=\emptyset$
for $i,j\in\N$  with $i\not=j$. This easily implies that
$\langle\{x_j:j\in\N\}\rangle=\bigoplus_{j\in\N} \langle x_j \rangle$
and
$K\cong\prod_{j\in\N} \langle x_j \rangle$.
Since $H\subseteq K$ by (i),
for every $h\in H$ there exists a
sequence $\{s_j:j\in\N\}$ of integer numbers so that
\begin{equation}
\label{representation:of:h}
h=\sum_{j=0}^\infty s_j x_j.
\end{equation}

\begin{claim}
$H$ is uniformly controllable in $G$.
\end{claim}
\begin{proof}
Fix an integer $n\in I$.
Apply (ii) to find $l\in \N$ such that
\begin{equation}
\label{condition:on:n}
n\le\max(\mathrm{supp} (x_l)).
\end{equation}
We claim that the integer
\begin{equation}
\label{condition:on:m}
m=\min(\mathrm{supp} (x_{l+1}))
\end{equation}
satisfies Definition \ref{uniform]controllability:in:Z}.
Clearly, $m\in I$.

Let $h,h'\in H$. Let \eqref{representation:of:h}
be the representation of $h$ and let
\begin{equation}
\label{representation:for:h'}
h'=\sum_{j=0}^\infty s_j' x_j
\end{equation}
be a similar representation of $h'$, for a suitable sequence
$\{s_j':j\in\N\}$ of integer numbers.

Since $x_0,\dots,x_l\in H$ and $H$ is a subgroup of $G$, we have
\begin{equation}
\label{eq:y:y'}
y=\sum_{j=0}^l s_j x_j\in H
\ \ \mbox{ and }\ \
y'=\sum_{j=0}^l s_j' x_j\in H.
\end{equation}
Since $h'=\sum_{j=0}^\infty s_j' x_j\in H$ and $H$ is a subgroup of $G$,
it follows that
\begin{equation}
\label{eq:z}
z=\sum_{j=l+1}^\infty s_j' x_j=h'-y'\in H.
\end{equation}
From $y\in H$ and $z\in H$, we conclude that $g=y+z\in H$.

Let $i\in I$ and $i\le n$.
From (ii) and \eqref{condition:on:n}, it follows that
$x_j(i)=0$ for all integers $j\ge l+1$, which yields
$$
h(i)=\sum_{j=0}^\infty s_j x_j(i)
=
\sum_{j=0}^l s_j x_j(i)
=
y(i)
$$
by
\eqref{representation:of:h}
and \eqref{eq:y:y'}.
Furthermore,
from \eqref{eq:z}
we get
$z(i)=0$.
Therefore,
$g(i)=y(i)+z(i)=y(i)=h(i)$.
This shows that
$g\restriction_{n^-\cap I}=h\restriction_{n^-\cap I}$.

Suppose now that $i\in I$ and $i\ge m$.
From (ii) and \eqref{condition:on:m}, it follows that
$x_j(i)=0$ for all integers $j=0,\dots,l$, which yields
$$
h'(i)=\sum_{j=0}^\infty s_j' x_j(i)
=
\sum_{j=l+1}^\infty s_j' x_j(i)
=z(i)
$$
by \eqref{representation:for:h'}
and \eqref{eq:z}.
Furthermore,
from \eqref{eq:y:y'}
we get
$y(i)=0$.
Therefore,
$g(i)=y(i)+z(i)=z(i)=h'(i)$.
This shows that
$g\restriction_{m^+\cap I}=h'\restriction_{m^+\cap I}$.
\end{proof}

Suppose now that
conditions (iii) and (iv) hold.
We are going to show that,
under these additional assumptions,
the group $H$ is not strongly controllable in $G$.

Indeed,
suppose the contrary. Then $H$ is
$k$-controllable in $G$ for some $k\in\N$;
see Definition \ref{classical:definition}.
Use (iii) to fix $l\in \N$ such that
$|\mathrm{supp} (x_l)|>k+1$.
Let $n=\min\mathrm{supp} (x_l)$
and $m=n+k$. Clearly, $n\in I$.
Note that
$x_l\in H$ by (i).
Since
$0\in H$
and $H$ is $k$-controllable in $G$,
there exists $g\in H$ such that
$g\restriction_{n^-\cap I}=x_l\restriction_{n^-\cap I}$
 and
$g\restriction_{m^+\cap I}=0$.
In particular, $g(n)=x_l(n)\not=0$,
as $n\in \mathrm{supp} (x_l)$.

Since $g\in H$, there exists a sequence of integers
$\{s_j:j\in\N\}$
such that
$g=\sum_{j=0}^\infty s_j x_j$.
From
$g\restriction_{n^-\cap I}=x_l\restriction_{n^-\cap I}$ and
(ii)
we conclude that $s_j x_j=0$
for all $j<l$.
Similarly,
from
$g\restriction_{m^+\cap I}=0$
and
(ii)
we conclude that $s_j x_j=0$
for all $j>l$.
This shows that $g=s_l x_l$.

Since $n=\min\mathrm{supp} (x_l)$
and $|\mathrm{supp} (x_l)|>k+1$,
we conclude that
$m=n+k<\max\mathrm{supp} (x_l)=i$.
Therefore,
$i\in m^+\cap I$, and so
$s_l x_l(i)=g(i)=0$.
Since
$i,n\in {supp} (x_l)$
and
$s_l x_l(n)=g(n)\not=0$,
this contradicts (iv).
\end{proof}

\begin{example}
Let $p$ be a prime number and let $G=\mathbb{Z}(p)^\Z$.
Choose any sequence $\{x_j:j\in\N\}\subseteq
\left(\bigoplus_{n\in\Z} \mathbb{Z}(p)\right)\setminus\{0\}$
satisfying conditions (ii) and
(iii) of Theorem
\ref{distinguishing:uniform:and:k},
in which we take $I=\Z$ and $G_i=\Z(p)$ for every $i\in\Z$.
Clearly, this sequence automatically satisfies the condition (iv) of
this theorem as well.
\begin{itemize}
\item[(i)]
{\em $K=\overline{\langle \{x_j:j\in\N\}\rangle}$ is a closed subgroup of $G$
which is uniformly controllable in $G$ but not strongly controllable in $G$\/}.
This follows from
Theorem \ref{distinguishing:uniform:and:k}, as $H=K$ satisfies the condition
(i) of this theorem.
\item[(ii)]
{\em $H=\langle \{x_j:j\in\N\}\rangle$ is a countable subgroup
of $G$ which is uniformly controllable in $G$ but not  strongly controllable in $G$\/}.
This also follows from Theorem \ref{distinguishing:uniform:and:k}, as $H$
satisfies the condition
(i) of this theorem.
\end{itemize}
\end{example}

\begin{theorem}
\label{unif:but:not:strongly}
Let $p$ be a prime number
and
let $\mathscr{H}_p$ be the family of subgroups of the product $G=\mathbb{Z}(p)^\N$ which are uniformly controllable in $G$ but not strongly controllable in $G$.
Then $|\mathscr{H}_p|=2^\cont$.
\end{theorem}
\begin{proof}
Choose any sequence $\{x_j:j\in\N\}\subseteq
\left(\bigoplus_{n\in\Z} \mathbb{Z}(p)\right)\setminus\{0\}$
satisfying the conditions (ii) and
(iii) of Theorem
\ref{distinguishing:uniform:and:k}.
Clearly, this sequence automatically satisfies the condition (iv) of
this theorem as well.
Note that $K=\overline{\langle \{x_j:j\in\N\}\rangle}$ is topologically isomorphic to
$\prod_{j\in\N}\langle \{x_j:j\in\N\}\rangle\cong \mathbb{Z}(p)^\N$,
so
$|K|=|\mathbb{Z}(p)^\N|=\cont$.
Therefore, we can fix a set $X\subseteq K$ such that
$|X|=\cont$ and
$\langle X\rangle\cap \langle \{x_j:j\in\N\}\rangle=\{0\}$.
Then the set $\mathscr{Y}$ of all subsets of $X$ has cardinality
$2^\cont$.

For every $Y\in\mathscr{Y}$, let $H_Y=\langle \{x_j:j\in\N\}\cup Y\rangle$
and note that $H_Y$ satisfies the condition (i) of Theorem
\ref{distinguishing:uniform:and:k}.
Applying Theorem
\ref{distinguishing:uniform:and:k}, we conclude that
each $H_Y$ is uniformly controllable but not strongly controllable in $G$,
so $H_Y\in \mathscr{H}_p$.
Since $H_Y\not=H_{Y'}$ whenever
$Y,Y'\in \mathscr{Y}$ and $Y\not= Y'$, it follows that
$|\mathscr{H}_p|\ge |\{H_Y:Y\in \mathscr{Y}\}|\ge |\mathscr{Y}|=2^\cont$.
Since $|K|=\cont$ and $H\subseteq K$ for every $H\in\mathscr{H}_p$, the reverse inequality $|\mathscr{H}_p|\le 2^\cont$ holds as well.
\end{proof}

\section{Compact products of countably many cyclic groups}

\label{profinite:section}

In this section we use the Pontryagin duality to obtain characterizations of compact products  of countably many cyclic groups.

Let $G$ be an arbitrary topological abelian group.
A character on $G$ is a continuous homomorphism from $G$ to the circle group $\T$.
The pointwise
sum
of two
characters is again a character, and the set $\widehat{G}$ of all characters
on $G$ is a group with pointwise addition as the composition law.
If $G$ is locally compact, then the group
$\widehat{G}$ equipped with the compact open topology becomes a
topological group called the
{\it Pontryagin dual group\/} of $G$. Many topological properties of a compact
abelian group $G$ can be
described via  algebraic  properties of the discrete dual
$\widehat G$ of $G$; see \cite{DPS, HM}.
For example, we have the following relations:
\begin{itemize}
\item $w(G)=|\widehat G|$;
\item $G$ is connected if and only if $\widehat G$ is torsion-free;
\item $G$ is profinite if and only if $\widehat G$ is torsion.
\end{itemize}

For a fixed prime number $p$, an abelian group $A$ is called a {\em $p$-group\/} if $A=\{x\in A: p^nx=0$ for some $n\in\N\}$.

\begin{lemma}
 \label{lemma:1}
Let $p\in\Prm$ and let $A$ be an abelian $p$-group.
If the torsion part $t(\du {A})$ of $\du{A}$ is dense in $\du{A}$,
then $\bigcap_{n\in\N} p^n A=\{0\}$.
\end{lemma}
\begin{proof}
Suppose that $a\in \bigcap_{n\in\N} p^n A$ and $a\not=0$.
Since $A$ is $q$-divisible for every $q\in\Prm\setminus\{p\}$,
it follows that $a\in\bigcap_{n\in\N^+} nA$.
For every $n\in\N$ fix $a_n\in A$ such that $n a_n=a$. Since $a\not=0$ and $A$ is discrete, there exists
$\chi_0\in\du{A}$ such that $\chi_0(a)\not=0$. Therefore,
$U=\{\chi\in\du{A}:\chi(a)\not=0\}$ is a non-empty open subset of $\du{A}$.
By our assumption, there exists $\chi\in t(\du{A})\cap U$.
Therefore, $n\chi=0$ for some $n\in\N^+$.
On the other hand, since $\chi\in U$ is a homomorphism,
$n\chi(a_n)=\chi(n a_n)=\chi(a)\not=0$,
which shows that $n\chi\not=0$, a contradiction.
\end{proof}

\begin{theorem}
\label{lemma:2}
If $G$ is a compact metric profinite abelian group such that $t(G)$ is dense in $G$, then
$G$ is topologically isomorphic to a product of countably many cyclic groups.
\end{theorem}
\begin{proof}
Since $G$ is compact metric abelian group, $A=\du{G}$ is a countable discrete group.
Since $G$ is profinite, $A$ is a torsion group.
Therefore, $A=\bigoplus_{p\in\Prm} A_p$, where each $A_p$ is a $p$-group.

Fix $p\in\Prm$.
Since $A_p$ is a direct summand of $A$,
there exists a continuous surjective homomorphism $f_p:G=\du{A}\to\du{A_p}$.
Since $t(G)$ is dense in $G$ and $f_p$ is continuous, $f_p(t(G))$ must be dense in
$\du{A_p}$. Since $f_p(t(G))\subseteq t(\du{A_p})$, we conclude that
$t(\du{A_p})$ is also dense in $\du{A_p}$.
Since $A_p$ is a $p$-group, from Lemma \ref{lemma:1} it follows that
$\bigcap_{n\in\N} p^n A_p=\{0\}$; that is,
$A_p$ does not have non-zero elements of infinite height.
Since $A_p$ is countable, by the second Pr\"ufer theorem, $A_p$ is a direct sum of (countably many) cyclic groups; see \cite{Fu}.

Since $A_p$ is a direct sum of cyclic groups for every $p\in\Prm$,
so is $A=\bigoplus_{p\in\Prm} A_p$.
By taking the dual, we conclude that $G=\du{A}$ is the direct product of
countably many cyclic groups.
\end{proof}

\begin{corollary}
\label{cor:8.3}
For an abelian group $G$, the following conditions are equivalent:
\begin{itemize}
\item[(i)] $G$ is topologically isomorphic to a direct product of countably many finite cyclic groups;
\item[(ii)] $G$ is topologically isomorphic to a direct product of countably many finite groups;
\item[(iii)] $G$ is zero-dimensional, compact metric and $t(G)$ is dense in $G$.
\end{itemize}
\end{corollary}
\begin{proof}
The implication (i)$\to$(ii)
is trivial.

(ii)$\to$(iii)
Suppose that $G\cong \prod_{n\in\N} F_n$, where $F_n$ is a finite group.
Then $G$ is clearly compact, zero-dimensional and metric.
Furthermore, $D=\bigoplus_{n\in \N} F_n$ is dense in $G$.
Since $t(G)$ contains $D$, it follows that
$t(G)$ is also dense in $G$.

(iii)$\to$(i)
By Theorem \ref{lemma:2},
$G\cong \prod_{n\in\N} C_n$, where each $C_n$ is a cyclic group.
Since each $C_n$ is compact, as a continuous image of the compact space $G$,
it must be finite.
\end{proof}

\begin{corollary}
\label{profinite:corollary}
For a compact metric abelian group $G$, the following conditions are equivalent:
\begin{itemize}
\item[(i)] $G$ is topologically isomorphic to a direct product of cyclic groups;
\item[(ii)] $G$ is profinite and $t(G)$ is dense in $G$.
\end{itemize}
\end{corollary}

\section{In conclusion}

\label{conclusion:section}

Numke studied in \cite{Nunke} closed subgroups of
the powers
$\Z^I$ of the integers
$\Z$ with the discrete topology.
For a countable set $I$, he showed that:
\begin{itemize}
\item
 every closed subgroup of
$\Z^I$ is itself a (direct) product,
\item
every endomorphism of $\Z^I$ is continuous, and
\item
every direct summand
of $\Z^I$ is closed (and therefore, is itself a product).
\end{itemize}

Our next corollary can be considered as an analogue of this theorem for countable products of finite cyclic groups.

\begin{corollary}
\label{product:corollary}
Let $I$ be a countable set, $\{G_i:i\in I\}$ be a family of finite abelian groups
and $G=\prod_{i\in I} G_i$ be its direct product.
Then every closed weakly controllable subgroup $H$ of $G$
is topologically isomorphic to a direct product
of finite cyclic groups.
\end{corollary}
\begin{proof}
Since all groups $G_i$ are finite,
$\bigoplus_{i\in I} G_i\subseteq t(G)$, and so
$H\cap \left(\bigoplus_{i\in I} G_i\right)\subseteq H\cap t(G)=t(H)$.
Since $H$ is weakly controllable in $G$,
$H\cap \left(\bigoplus_{i\in I} G_i\right)$ is dense in $H$.
It follows that $t(H)$ is also dense in $H$.
Clearly, $G$ is a compact profinite metric group, and so is its closed subgroup $H$.
Applying the implication (ii)$\to$(i) of
Corollary \ref{profinite:corollary}
to $H$, we conclude that $H$ is topologically isomorphic to a direct product $\prod_{j\in J} C_j$ of cyclic groups $C_j$.
Since each $C_j$ is a closed subgroup of the compact group $H$, it is also compact. Therefore, each $C_j$ must be finite.
\end{proof}

Since direct products of finite groups are compact, the closedness of $H$ in $G$
is a necessary condition for $H$ to be topologically isomorphic to a direct product of finite groups. The next example shows that Corollary \ref{product:corollary}
fails without the assumption that $H$ is weakly controllable in $G$.

\begin{example}
Let $p$ be a prime number. For every $i\in\N$ let $G_i=\Z(p^i)$ be the cyclic group  of order $p^i$. Then the direct product $G=\prod_{i\in\N} G_i$ contains a closed subgroup $H$ topologically isomorphic to the group $\Z_p$ of $p$-adic integers.
(Indeed, use the fact that $\Z_p$ is topologically isomorphic to the limit of the inverse sequence of $G_i=\Z(p^i)$; see \cite{DPS}.)
On the other hand, $\Z_p$ is known to be indecompasable; that is, $\Z_p$ cannot be represented as a product of two topological groups; see \cite{DPS}.
\end{example}

Recall that a compact group $G$ is a {\em Valdivia compact group\/} provided that
$G$ is homeomorphic to the subspace
$$
\{x\in [0,1]^I: \{i\in I: x(i)\not=0\}
\
\mbox{ is countable}
\}
$$
of
the
Tychonoff cube
$[0,1]^I$ for some index set $I$; see \cite{KubisUspenskij:2005}.
Several important results have been proved for these groups recently \cite{Kubis:2008,
Chigogidze:2008,KubisUspenskij:2005}.
Chigogidze proves in \cite{Chigogidze:2008}
that every Valdivia compact  group is homeomorphic to a product of metrizable compacta.
Kubis \cite{Kubis:2008}
considers the smallest class $\mathcal R$  of compact spaces containing all compact metric spaces
and closed under limits of continuous inverse sequences of retractions, and he shows that every compact connected
abelian group which is a topological retract of a space from class $\mathcal R$ is
homeomorphic
to a product of compact
metric spaces.

Our
Corollary \ref{product:corollary} could be compared with the above results.
It seems worth noting that in Corollary \ref{product:corollary} one has a topological isomorphism, while results mentioned above give only
a homeomorphism (discarding the algebraic structure).

\begin{remark}
Apparently, Staiger \cite{staiger} was the first to suggest that taking closed
subgroups of direct products is crucial for obtaining good properties of group codes they represent.
Most (although not all) of
the examples in our paper deal with non-closed subgroups of direct products.
This provides a strong evidence in support of the validity of Staiger's hypothesis.
\end{remark}

For convenience of specialists in coding theory,
we restate here
some of our results
in terms commonly used in coding theory.

\begin{corollary}
\label{th_summary}
Let $\mathcal C$ be a complete group code in $G=\prod\limits_{i\in \N} G_i$,
where
every group $G_i$ is
finite.
Then the following conditions are equivalent:
\begin{enumerate}
\item $\mathcal C$ is weakly controllable;

\item $\mathcal C$ is controllable;

\item $\mathcal C$ is uniformly controllable.
\end{enumerate}
Furthermore, if all $G_i$ are abelian, then $\mathcal C$ is
topologically isomorphic to a direct product of finite cyclic groups.
\end{corollary}
\begin{proof}
A ``complete group code in $G$'' means ``closed subgroup of $G$'', so
$\mathcal C$ is a closed subgroup of $G$.
Therefore, the equivalence of (1), (2) and (3) follows from
Corollary \ref{uniform:controllability}.
The final statement is proved in Corollary \ref{product:corollary}.
\end{proof}

\begin{remark}
When the code
$\mathcal C$ in
Corollary \ref{th_summary}
is in addition time invariant, then it follows  from condition (3)
that all properties in this corollary are equivalent to the strong controllability
of $\mathcal C$; see \cite{fagnani:adv97}.
\end{remark}

\end{document}